\documentclass{amsart}
\usepackage{amsopn}
\usepackage{amssymb, amscd, amsmath}
\usepackage{array}
\usepackage{color}

\newcommand{\nc}{\newcommand}

\DeclareRobustCommand{\SkipTocEntry}[4]{}

\nc{\eg}{\mathfrak{e}}
\nc{\detb}{\det_{\Beta_\mg}}
\nc{\Slb}{\Sl_{\Beta}}
\nc{\Slbm}{\Sl_{\Beta_\mg}^H}
\nc{\slbm}{\slg_{\Beta_\mg}^H(n)}
\nc{\ggl}{\mathfrak{g}}
\nc{\ppm}{\mathfrak{p}}
\nc{\GG}{G}
\nc{\Gm}{\Gl(\mg)}
\nc{\Gg}{\Gl(\ggl)}
\nc{\glgg}{{\mathfrak{gl}(\mathfrak{g})}}
\nc{\GH}{\Gl^{\!H}}
\nc{\GHm}{\Gl^{\!H} (\mg)}
\nc{\GHmk}{\Gl^H(\mg_k)}
\nc{\gHm}{{\mathfrak{gl}^H(\mathfrak{m})}}
\nc{\OHm}{\Or^H(\mg)}
\nc{\sogHm}{ {\mathfrak{so}^H(\mathfrak{m})} }
\nc{\Vm}{V(\mg)}
\nc{\Vg}{V(\ggl)}
\nc{\Vmi}{V(\mg_\infty)}
\nc{\Om}{\Or(\mg)}
\nc{\Og}{\Or(\ggl)}
\nc{\Symm}{{\rm Sym}(\mg)}
\nc{\Symg}{{\rm Sym}(\ggl)}
\nc{\gm}{\mathfrak{gl}(\mg)}
\nc{\som}{\sog(\mg)}
\nc{\sogg}{{\mathfrak{so}(\mathfrak{g})}}
\nc{\hml}{{\mu^{\ggl}}}
\nc{\ml}{{\mu^{\ggl}_{\mg}}}
\nc{\Ol}{{\mathcal{O}_{\ggl}}}
\nc{\OlH}{{\mathcal{O}^H_{\ggl}}}
\nc{\Vn}{V(n)}
\nc{\SymV}{{\rm Sym}(V)}
\nc{\mub}{{\bar \mu}}
\nc{\mumg}{{\mu_\mg}}
\nc{\Betam}{{\Beta_\mg}}
\nc{\Betag}{{\Beta_\ggo}}
\nc{\Betagp}{{\Beta_\ggo^+}}
\nc{\ii}{{\mathrm{i}}}
\nc{\Nrm}{{\mathrm{N}}}
\nc{\Srm}{{\mathrm{S}}}
\nc{\spec}{{\operatorname{sp}}}
\nc{\iR}{{\ii\RR}}

\nc{\Vs}{{V(\sg)}}
\nc{\Syms}{{\Sym(\sg)}}
\nc{\glgs}{{\glg(\sg)}}

\nc{\slgb}{{\slg_\Beta}}

\nc{\Vzerog}{V_{\Beta_\ggo^+}^{0}}
\nc{\Vnng}{V_{\Beta_\ggo^+}^{\geq 0}}
\nc{\Vnnssg}{U_{\Beta_\ggo^+}^{\geq 0}}
\nc{\Vzerossg}{U_{\Beta_\ggo^+}^{0}}

\nc{\musol}{{\mu_{\mathsf{sol}}}}

\nc{\Gl}{\mathsf{GL}} \nc{\Or}{\mathsf{O}}  \nc{\SO}{\mathsf{SO}}   \nc{\Sl}{\mathsf{SL}}  
\nc{\G}{\mathsf{G}} \nc{\K}{\mathsf{K}}  \nc{\T}{\mathsf{T}} \nc{\Lsf}{\mathsf{L}}
\nc{\Qb}{\mathsf{Q}_\Beta} \nc{\Hb}{\mathsf{H}_\Beta} \nc{\Ub}{\mathsf{U}_\Beta} 
\nc{\Gb}{\mathsf{G}_\Beta} \nc{\Kb}{\mathsf{K}_\Beta} \nc{\Hh}{\mathsf{H}}
\nc{\PPP}{\mathsf{P}}  
\nc{\Ss}{\mathsf{S}}
\nc{\U}{\mathsf{U}}

\nc{\Gs}{{\Gl(\sg)}}  \nc{\Os}{{\Or(\sg)}}
\nc{\gsol}{{g_{\mathsf{sol}}}}
\nc{\bgsol}{{\bar g_{\mathsf{sol}}}}

\nc{\GGs}{S}
\nc{\ggs}{\mathfrak{s}}

 \nc{\ggo}{\mathfrak{g}}
 \nc{\ggob}{\overline{\mathfrak{g}}}
\nc{\lamg}{\Lambda^2\ggo^*\otimes\ggo}
\nc{\gkp}{(\ggo=\kg\oplus\pg,\ip)} \nc{\ukh}{(\ug=\kg\oplus\hg,\ip)}
\nc{\tgkp}{(\tilde{\ggo}=\kg\oplus\pg,\ip)}

\nc{\fg}{\mathfrak{f}}  \nc{\vg}{\mathfrak{v}} \nc{\wg}{\mathfrak{w}} \nc{\zg}{\mathfrak{z}} \nc{\ngo}{\mathfrak{n}} \nc{\kg}{\mathfrak{k}} \nc{\mg}{\mathfrak{m}} \nc{\bg}{\mathfrak{b}}  \nc{\sog}{\mathfrak{so}} \nc{\sug}{\mathfrak{su}} \nc{\spg}{\mathfrak{sp}} \nc{\slg}{\mathfrak{sl}} \nc{\glg}{\mathfrak{gl}} \nc{\cg}{\mathfrak{c}} \nc{\rg}{\mathfrak{r}}  \nc{\hg}{\mathfrak{h}} \nc{\tgo}{\mathfrak{t}} \nc{\ug}{\mathfrak{u}} \nc{\dg}{\mathfrak{d}} \nc{\ag}{\mathfrak{a}} \nc{\pg}{\mathfrak{p}} \nc{\sg}{\mathfrak{s}} \nc{\affg}{\mathfrak{aff}} \nc{\qg}{\mathfrak{q}}
\nc{\Xg}{\mathfrak{X}} \nc{\lgo}{\mathfrak{l}} \nc{\tg}{\mathfrak{t} }

\nc{\pca}{\mathcal{P}} \nc{\nca}{\mathcal{N}} \nc{\lca}{\mathcal{L}} \nc{\oca}{\mathcal{O}} \nc{\mca}{\mathcal{M}} \nc{\tca}{\mathcal{T}} \nc{\aca}{\mathcal{A}} \nc{\cca}{\mathcal{C}} \nc{\gca}{\mathcal{G}} \nc{\sca}{\mathcal{S}} \nc{\hca}{\mathcal{H}} \nc{\bca}{\mathcal{B}} \nc{\dca}{\mathcal{D}} \nc{\fca}{\mathcal{F}} \nc{\Qca}{\mathcal{Q}}

\nc{\dd}{{\rm d}}  \nc{\ddt}{\tfrac{{\rm d}}{{\rm d}t}}        \nc{\dds}{\tfrac{{\rm d}}{{\rm d}s}} 
\nc{\ddtbig}{\frac{{\rm d}}{{\rm d}t}}      \nc{\dpar}{\tfrac{\partial}{\partial t}}    

\nc{\im}{\mathtt{i}} \renewcommand{\Re}{{\rm Re}}   

\nc{\RR}{{\mathbb R}} \nc{\HH}{{\mathbb H}} \nc{\CC}{{\mathbb C}} \nc{\ZZ}{{\mathbb Z}}
\nc{\FF}{{\mathbb F}} \nc{\NN}{{\mathbb N}} \nc{\QQ}{{\mathbb Q}} \nc{\PP}{{\mathbb P}}
\nc{\KK}{{\mathbb K}}

\nc{\vs}{\vspace{.2cm}} \nc{\vsp}{\vspace{1cm}} 
\nc{\ip}{{\langle \,\cdot \,,\cdot \,\rangle }}
 \nc{\la}{\langle} \nc{\ra}{\rangle} \nc{\unm}{\tfrac{1}{2}}
\nc{\unc}{\tfrac{1}{4}} \nc{\und}{\tfrac{1}{16}} \nc{\no}{\vs\noindent}
\nc{\lam}{\Lambda^2(\RR^n)^*\otimes\RR^n} \nc{\tangz}{{\rm T}^{\rm Zar}}

\nc{\nor}{{\sf n}}  \nc{\mum}{/\!\!/} \nc{\kir}{/\!\!/\!\!/}
\nc{\Ri}{\tfrac{4\Ric_{\mu}}{||\mu||^2}} \nc{\ds}{\displaystyle}
\nc{\ben}{\begin{enumerate}} \nc{\een}{\end{enumerate}} \nc{\f}{\frac}
\nc{\lb}{[\cdot,\cdot]} \nc{\isn}{\tfrac{1}{||v||^2}}

\nc{\wt}{\widetilde}
\nc{\raw}{\rightarrow} \nc{\lraw}{\longrightarrow} \nc{\hqn}{\mathcal{H}_{q,n}}
\nc{\minimatrix}[4]{\left[\begin{smallmatrix} {#1} & {#2} \\ {#3} & {#4} \end{smallmatrix}\right]}
\nc{\twomatrix}[4]{\left[\begin{array}{cc} {#1} & {#2} \\ {#3} & {#4} \end{array} \right]}
\nc{\threematrix}[9]{\left[\begin{array}{ccc} {#1} & {#2} & {#3} \\ {#4} & {#5} & {#6}\\ {#7} & {#8} & {#9} \end{array} \right]}
\nc{\mut}{\tilde{\mu}} \nc{\mur}{{\mu_r}} \nc{\mutr}{{\tilde{\mu}_r}}

\nc{\alert}{\color{blue}}

\nc{\glgan}{\minimatrix{0}{0}{\star}{0}} \nc{\glgna}{\minimatrix{0}{\star}{0}{0}}  \nc{\glgnn}{\minimatrix{0}{0}{0}{\star}}  \nc{\glgaa}{\minimatrix{\star}{0}{0}{0}}
\nc{\Vaan}{{\left(\ag \wedge \ag\right)^* \otimes \ngo}} \nc{\Vann}{{\left(\ag \otimes \ngo \right)^* \otimes \ngo}} \nc{\Vnnn}{{\left(\ngo \wedge \ngo \right)^* \otimes \ngo}}

\nc{\ad}{\operatorname{ad}}  \nc{\Aut}{\operatorname{Aut}}   \nc{\Inn}{\operatorname{Inn}}   \nc{\Lie}{\operatorname{Lie}} \nc{\Ad}{\operatorname{Ad}} \nc{\Der}{\operatorname{Der}} \nc{\rad}{\operatorname{rad}} \nc{\kf}{\operatorname{B}}

\nc{\End}{\operatorname{End}} \nc{\rank}{\operatorname{rank}} \nc{\Ker}{\operatorname{Ker}} \nc{\tr}{\operatorname{tr}} \nc{\Sym}{\operatorname{Sym}} \nc{\diag}{\operatorname{diag}} \nc{\proy}{\operatorname{pr}} \nc{\Adj}{\operatorname{Adj}} \nc{\proj}{\operatorname{pr}} \nc{\Id}{{\operatorname{Id}}} \nc{\Span}{\operatorname{span}}

\nc{\Hess}{\operatorname{Hess}}  \nc{\dif}{\operatorname{d}} \nc{\sen}{\operatorname{sen}} \nc{\grad}{\operatorname{grad}} \nc{\Order}{\operatorname{O}} \nc{\divg}{\operatorname{div}}

\nc{\Iso}{\operatorname{Iso}} \nc{\Diff}{\operatorname{Diff}} \nc{\Rc}{\operatorname{Rc}} \nc{\Ricci}{\operatorname{Ric}}
\nc{\ric}{\operatorname{ric}} 
\nc{\Riem}{\operatorname{Rm}} \nc{\scal}{\operatorname{scal}} \nc{\scalm}{\operatorname{scal}^\star} \nc{\Riccim}{\operatorname{Ric}^{\star}} \nc{\tang}{\operatorname{T}} \nc{\vol}{\operatorname{vol}} \nc{\mcv}{\operatorname{H}} \nc{\inj}{\operatorname{inj}}
\nc{\isog}{\mathfrak{iso}}

\nc{\mm}{\operatorname{M}} \nc{\CH}{\operatorname{CH}} \nc{\Irr}{\operatorname{Irr}} \nc{\mcc}{\operatorname{mcc}} \nc{\Sb}{\mathcal{S}_\Beta} \nc{\mmm}{\operatorname{m}} 

\nc{\Symn}{{{\rm Sym}(n)}}
\nc{\Beta}{{\beta}}
\nc{\Alpha}{A}
\nc{\Vr}{V_{\Beta^+}^{r}}
\nc{\Vzero}{V_{\Beta^+}^{0}}
\nc{\Vnn}{V_{\Beta^+}^{\geq 0}}
\nc{\Vnnss}{U_{\Beta^+}^{\geq 0}}
\nc{\Vzeross}{U_{\Beta^+}^{0}}
\nc{\Vnnt}{V_{\tilde\Beta^+}^{\geq 0}}
\nc{\Betap}{ {\Beta + \Vert{\Beta}\Vert^2 \Id} }
\nc{\Ap}{ {A + \Vert{A}\Vert^2 \Id} }
\nc{\zero}{ {\backslash \{0\} } }
\nc{\normmm}{{\rm F}}
\nc{\ipp}{\la\,\cdot \,,\cdot\,\ra^*_g}
\nc{\ippk}{\la\,\cdot \,,\cdot\,\ra^*_{g_k}}
\nc{\ippi}{\la\,\cdot \,,\cdot\,\ra^*_{g_\infty}}
\nc{\ipnew}{\la \la \cdot , \cdot \ra\ra}
\nc{\der}{\mathfrak{der}}
\nc{\kfm}{\widetilde{\kf}} 
\nc{\KFm}{\widetilde{\mathcal{B}}}
\nc{\KF}{\mathcal{B}}
\nc{\II}{{\mathbb I}}
\nc{\spa}{\operatorname{span}}

\theoremstyle{plain}
\newtheorem{theorem}{Theorem}[section]
\newtheorem{proposition}[theorem]{Proposition}
\newtheorem{corollary}[theorem]{Corollary}
\newtheorem{lemma}[theorem]{Lemma}
\newtheorem{teointro}{Theorem}
\newtheorem{corintro}[teointro]{Corollary}

\theoremstyle{definition}
\newtheorem{definition}[theorem]{Definition}

\theoremstyle{remark}
\newtheorem{remark}[theorem]{Remark}

\begin{document}
\begin{titlepage}

\title{The Ricci flow on solvmanifolds of real type}

\author{Christoph B\"ohm}	
\address{University of M\"unster, Einsteinstra{\ss}e 62, 48149 M\"unster, Germany}
\email{cboehm@math.uni-muenster.de}

\author{Ramiro A.~ Lafuente} 
\address{University of M\"unster, Einsteinstra{\ss}e 62, 48149 M\"unster, Germany}
\email{lafuente@uni-muenster.de}

\thanks{The second author was supported by the Alexander von Humboldt Foundation.}

\begin{abstract}
We show that for any solvable Lie group of real type, 
any homogeneous Ricci flow solution converges in Cheeger-Gromov topology to a unique non-flat solvsoliton, which is independent of the initial left-invariant metric. As an application, we obtain results on the isometry groups of non-flat solvsoliton metrics and Einstein solvmanifolds.
\end{abstract}

\end{titlepage}


\maketitle
\setcounter{page}{1}
\setcounter{tocdepth}{0}

%


\section{Introduction}

A \emph{solvmanifold} is a simply-connected solvable Lie group $\Ss$
 endowed with a left-invariant Riemannian metric. 
 It is called of \emph{real type}, if $\Ss$ is non-abelian and if for each element of its Lie algebra the corresponding adjoint map is either
 nilpotent or  has an eigenvalue with non-zero real part.
Nilmanifolds are of real type, 
and so are --up to isometry-- 
all homogeneous manifolds with negative sectional curvature, see \cite{Hei74} and \cite{Jbl13c}, but flat solvmanifolds are  not  \cite{Mln}. More interestingly,
by the deep structure results  of \cite{Heber1998}, \cite{standard}  and \cite{solvsolitons},
examples of solvmanifolds of real type include all non-flat
Einstein solvmanifolds and \emph{solvsolitons}.  
Recall that a solvsoliton is a solvmanifold which
is also an expanding Ricci soliton, and whose corresponding Ricci flow evolution
is  driven by  diffeomorphisms which are Lie group automorphisms.
Thus, up to isometry, solvmanifolds of real type
contain all known examples of non-compact, non-flat homogeneous Ricci solitons.

Since the homogeneous Ricci flow solution starting at a solvmanifold exists for all positive times \cite{scalar},  any sequence of blow-downs subconverges to a 
homogeneous limit Ricci soliton \cite{BL17}.
Our first main result addresses the question of uniqueness of such limits:

\begin{teointro}\label{mainthm_uniq}
 On a sim\-ply-connec\-ted solvable Lie group $\Ss$ of real type,
any scalar-curvature-normalized homogeneous Ricci flow solution converges in Cheeger-Gromov topology to a non-flat solvsoliton  $\big(\bar \Ss, \bgsol\big)$, which does not 
 depend on the initial metric.
\end{teointro}

The Lie group $\Ss$ is of course called of real type, if it satisfies the above condition.
Notice, that the limit transitive group $\bar \Ss$ may be non-isomorphic to $\Ss$, but must still be of real type:
see 
 Remark \ref{rem_solvsolreal} and Remark \ref{rem_realtype}.
Theorem \ref{mainthm_uniq} was known for $\dim \Ss = 3$ and partially for $\dim \Ss = 4$  \cite{Lot07}, for nilpotent Lie groups \cite{nilRF}, \cite{Jbl11}, and for unimodular, almost-abelian Lie groups \cite{Arroyo2013}.
In the compact case, such a uniqueness result does not hold in general,
since most compact Lie groups  admit non-isometric Einstein metrics.

A first immediate consequence of Theorem \ref{mainthm_uniq}  is

\begin{corintro}\label{maincor_globalstability}
Let $(\Ss,\gsol)$ be a non-flat solvsoliton. Then,  any scalar-curvature-normalized homogeneous Ricci flow solution on $\Ss$ converges in 
Cheeger-Gromov topology to $(\Ss,\gsol)$.
\end{corintro}

We not only recover that solvsolitons on such $\Ss$ 
are unique up to isometry, 
proved in \cite{solvsolitons}, but show also 
that homogeneous Ricci solitons on a solvable Lie group of real type are pairwise
equivariantly isometric: see  Corollary \ref{cor_uniqsolvsolmetric}.

 
 Recall that along a homogeneous Ricci flow solution
  the full isometry group remains unchanged \cite{Kot}. 
However, this does not imply in general  that this group
  will be a subgroup of the full isometry group of the limit: see \cite[Example 1.6]{GJ15}. The main reason for this is that, in the non-Einstein solvsoliton case, 
we only have convergence in Cheeger-Gromov topology: see Remark \ref{rem:solnonCinfty}.
Still, Theorem \ref{mainthm_uniq}  yields

\begin{corintro}
For a non-flat solvsoliton $(\Ss,\gsol)$,
the dimension of its isometry group $\Iso(\Ss,\gsol)$  is maximal among all left-invariant metrics on $\Ss$.
\end{corintro}


In the Einstein case, the convergence can be improved as follows:

\begin{teointro}\label{mainthm_Einstein}
Let $(\Ss,g_E)$ be a non-flat Einstein solvmanifold. Then,
any scalar-curvature-normalized homogeneous Ricci flow solution on $\Ss$ converges in $C^\infty$ topology to $\psi^* g_E$, for some $\psi \in \Aut(\Ss)$.
\end{teointro}

Notice that Theorem \ref{mainthm_Einstein} gives in particular a dynamical proof of the recent result of Gordon and Jablonski on the maximal symmetry of Einstein solvmanifolds \cite{GJ15}. 

There exist already in the literature several results on the stability of certain non-compact homogeneous
Einstein metrics and Ricci solitons, see for instance \cite{SSS11}, 
 \cite{Bam15}, \cite{JPW16}, \cite{WW16}. Even though more general, non-homogeneous variations are considered in these articles, none of them implies Theorem \ref{mainthm_Einstein}, since two different homogeneous metrics on $\RR^n$ are not within bounded distance to each other.

Our last main result provides a geometric characterization of solvmanifolds of real type, in terms of the behavior of homogeneous Ricci flow solutions. More precisely, recall that given a Cheeger-Gromov-convergent sequence $(M^n_k,g_k)_{k\in\NN}$ of homogeneous manifolds, each of which having an $N$-dimensional transitive group of isometries $\G_k$, there is a natural way of making sense of an $N$-dimensional limit group of isometries $\bar \G$, by taking limits of appropriately rescaled Killing fields 
(see \cite[$\S$9]{BL17}). The sequence is then called \emph{algebraically non-collapsed} if the action of $\bar \G$ is transitive on the limit space $(\bar M^n, \bar g)$, and collapsed otherwise. A homogeneous Ricci flow solution is algebraically non-collapsed if any convergent sequence of parabolic blow-downs has that property.

\begin{teointro}\label{mainthm_algnonc}
On a simply-connected, non-abelian, solvable Lie group $\Ss$, a homogeneous Ricci flow solution  is algebraically non-collapsed if and only if $\Ss$ is of real type.
\end{teointro}

A first consequence of Theorem \ref{mainthm_algnonc} is that for Ricci flow solutions on a simply-connected, non-abelian solvable Lie group which are \emph{not} of real type,  the dimension of the isometry group must  always jump in the limit: see Section \ref{sec_nocollapse}.
But even more importantly, studying  Ricci flow solutions on such Lie groups
would in general involve the understanding of algebraic collapse, 
which cannot be achieved by using the moving brackets framework only. 


We turn to the content of the paper and
the proofs of the above results.  In Section \ref{sec_real} we 
discuss algebraic properties of solvmanifolds of real type. In Section
\ref{sec_stratif} we recall the GIT stratification of the space of brackets
and state in Theorem \ref{thm_GITuniq} a uniqueness result
for critical points of the energy map. 
 This is the main ingredient in the proof Theorem \ref{thm_uniqsolvsol}, a uniqueness result for solvsoliton brackets lying in the intersection of the closure of an orbit and the stratum containing that orbit. Finally,
using the equivalence of the Ricci flow and the gauged bracket flow,
we show in Section \ref{sec_thmA} that on a solvmanifold of real type 
any scalar curvature normalized bracket flow converges to a unique solvsoliton bracket in the closure of the corresponding orbit. The proof uses essentially
a Lyapunov function  for the bracket flow, described in \cite{BL17}.
Then Theorem \ref{mainthm_uniq} follows immediately. The proof of Theorem \ref{mainthm_algnonc} is given in Section \ref{sec_nocollapse}. Finally, in Section \ref{sec_Einstein} we prove Theorem \ref{mainthm_Einstein},
using the computations of the linearization of the gauged bracket flow
at an Einstein bracket given in Section \ref{sec_linear}.

\section{Solvable Lie groups of real type}\label{sec_real}

In this section we discuss algebraic properties of  solvable Lie groups or real type.
Let  $(\sg, \mu)$ be the Lie algebra of a solvable Lie group $\Ss$. The Lie bracket
$\mu$ is a skew-symmetric bilinear map, i.e.~ an element of the vector space 
of \emph{brackets}
\begin{equation}\label{eqn_brackets}
	\Vs := \Lambda^2 \sg^* \otimes \sg,
\end{equation}
 that satisfies the Jacobi identity. We denote by
 $\ad_\mu(X)(Y) := \mu(X, Y) \in \sg$, $X,Y\in \sg$, the 
 corresponding adjoint maps.
For convenience, we also fix a scalar product $\ip$ on $\sg$.

\begin{definition}\label{def_realtype}
A non-abelian solvable Lie algebra $(\sg,\mu)$ is called
\begin{itemize}
  \item[(i)] of \emph{imaginary type}, if $\spec( \ad_\mu(X)) \subset i \, \RR$ for all $X\in \sg$;
  \item[(ii)] of \emph{real type}, if all its subalgebras of imaginary type are nilpotent.
\end{itemize}
A solvable Lie group is called of imaginary or real type if its Lie algebra is of that type.
\end{definition}

In the above definition,  $\spec(E)$ denotes
 the spectrum of an endomorphism $E$ of $\sg$. 
Notice that $\sg$ is of real type if and only if all for all  $X\in \sg$ the adjoint map
 $\ad_\mu(X)$ is either nilpotent or has an eigenvalue $\lambda \notin i \RR$.
Recall also that there exist non-abelian solvable Lie groups admitting flat left-in\-variant metrics. 
By \cite{Mln}, they are all of imaginary type and have abelian nilradical; we call the corresponding Lie brackets \emph{flat brackets}.

\begin{remark}\label{rem_solvsolreal}
In \cite{Jbl13c}, Lie algebras of real type are called \emph{almost completely solvable}.
Completely solvable Lie algebras, characterized by  $\spec ( \ad(X) ) \subset \RR$ for all $X\in \sg$, are of course of real type. Moreover, by \cite[Prop.8.4]{Jbl2015},
any solvable Lie group $\Ss$ admitting a solvsoliton metric is of real type.
\end{remark}

 Considering $\sg$ as a real vector space, there is a natural `change of basis' linear action of the group $\Gl(\sg)$ on $\Vs$, given by
\begin{equation}\label{eqn_Gsaction}
		\big(h \cdot \mu\big)(\cdot, \cdot) := h \mu (h^{-1} \, \cdot, h^{-1} \, \cdot), \qquad h\in \Gl(\sg), \quad \mu \in \Vs\,.
\end{equation}
The orbit $\Gl(\sg) \cdot \mu$ is precisely the set of brackets on $\sg$ such that the corresponding Lie algebra is isomorphic to $(\sg,\mu)$. Since 
by (\ref{eqn_Gsaction}) we have 
\begin{eqnarray}\label{eqn_adconj}
\ad_{h\cdot \mu} (X)= {h \ad_\mu(h^{-1}X)h^{-1}}\,,
\end{eqnarray}
the type of a Lie bracket is constant on each $\Gl(\sg)$-orbit.
Next,  for $\mu \in V(\sg)$, $X\in \sg$ we set
\begin{align*}
  \varphi(\mu,X) = \max \big\{ \, |\Re(\lambda) | : \lambda \in \spec( \ad_\mu(X) ) \big\}\,,
  \quad
  \psi(\mu,X) = \max \big\{ \, |\lambda | : \lambda \in \spec(\ad_\mu(X) ) \big\}\,.
\end{align*}

\begin{lemma}\label{lem_solveigenvaluesad}
Let $(\sg,\mu_0)$ be a solvable Lie algebra with nilradical $\ngo$, and let $\ag$ be the orthogonal complement
of $\ngo $ in $\sg$. If $\mu =h \cdot \mu_0$ for $h \in \Gl(\sg)$, then
for all $X\in \sg$ we have
\begin{equation*}
  \varphi(h \cdot \mu_0, X) =  \varphi(\mu_0, (h^{-1} X)_\ag )\quad \textrm{ and }\quad
   \psi(h \cdot \mu_0, X) =  \psi(\mu_0, (h^{-1} X)_\ag )\,,
 \end{equation*}
 where the $\ag$-subscript denotes orthogonal projection onto $\ag$.
Moreover, if $\ngo \neq \sg$ then $\mu_0$ is of real type if and only if $\sigma_\ag(\mu_0) := \min_{X\in \ag, \Vert X \Vert = 1}  \varphi (\mu_0,X) > 0$.
\end{lemma}

\begin{proof}
By \eqref{eqn_adconj} we have  $\varphi(h \cdot \mu_0, X)=
\varphi(\mu_0,h^{-1}\cdot X)$ and we obtain the first two claims, since 
 for any $X\in \sg$ and  $Y\in \ngo$ the adjoint maps 
 $\ad_{\mu_0}(X+Y)$ and $\ad_{\mu_0}(X)$ share the same eigenvalues.
 This follows from \cite[Theorem 3.7.3]{Varad84}, since there is a basis for the complexified 
 Lie algebra $\sg^\CC$ such that all the adjoint maps are upper triangular, and strictly upper
  triangular for $Y\in \ngo$.
 
 Regarding the second claim, $\varphi$ is continuous in
 $X$ and $\mu$, because the eigenvalues of linear maps vary continuously. Since for $X \in \ag \setminus \{ 0 \}$ the endomorphism $\ad_{\mu_0}(X)$ is not nilpotent, it immediately follows that $\mu_0$ is not of real type if and only if 
 $\sigma_\ag(\mu_0) = 0$.  
\end{proof}

\begin{lemma}\label{lem_realnonflat}
Let $\mu_0\in V(\sg)$ be a solvable Lie bracket of real type. Then, there are no non-zero flat brackets in $\overline{\Gl(\sg) \cdot \mu_0}$.
\end{lemma}

\begin{proof}
 Since $(\sg,\mu_0)$ is of real type, 
 $\varphi(\mu_0,Y) > 0$ for all $Y\in \ag\backslash\{0 \}$. Therefore, 
 using that the continuous
 maps $(\mu,X)\mapsto \varphi(\mu,X),\,\, \psi(\mu,X)$ are homogeneous of degree one in $X$, 
 we see that there exists a constant $C_{\mu_0}>0$ such that 
 $ \psi(\mu_0,Y) \leq C_{\mu_0} \cdot \varphi(\mu_0, Y)$
 for all $Y\in \ag$.

 Suppose now that $h_k\cdot \mu_0$ converges for $k \to \infty$
 to a flat bracket $\bar \mu \in V(\sg)$. Recall, that by
  \cite{Mln} the bracket $\bar \mu$ is of imaginary type, that is
 $\varphi(\bar \mu, \cdot ) \equiv 0$.
 For any fixed $X\in \sg$ 
 we then obtain by Lemma  \ref{lem_solveigenvaluesad}
\[
	\psi(h_k \cdot \mu_0, X) = \psi(\mu_0, (h_k^{-1} X)_\ag) \leq C_{\mu_0} \cdot \varphi(\mu_0, (h_k^{-1} X)_\ag) = C_{\mu_0} \cdot \varphi(h_k\cdot \mu_0, X) \underset{k\to\infty}\longrightarrow 0\,.
\]
 Therefore, $\psi(\bar \mu,  \cdot ) \equiv 0$ and $\bar \mu$ is nilpotent by Engel's theorem.
 Thus $\bar \mu = 0$ by \cite{Mln}.
\end{proof}

Recall, that the \emph{rank} of a solvable Lie algebra is the codimension of its nilradical.

\begin{lemma}\label{lem_realtype}
Among solvable Lie brackets of fixed rank, those of real type form an open set. 
\end{lemma}

\begin{proof}
If the rank is zero, the claim is obvious.
Let now $(\mu_k)_{k\in\NN} \in \Vs$ be a sequence of rank $a \in \NN$
 solvable Lie brackets which are not of real type, converging to $\mu_0$ as $k\to \infty$, also of rank $a$. After acting with suitable orthogonal maps on each $\mu_k$ and passing to a convergent subsequence, we may assume that the nilradicals of $\mu_k, \mu_0$ are a fixed subspace 
$\ngo \neq \sg$. The claim now follows by Lemma \ref{lem_solveigenvaluesad},
since $0=\lim_{k\to \infty}\sigma_{\ag}(\mu_k)=\sigma_{\ag}(\mu_0)$.
\end{proof}


\begin{remark}\label{rem_realtype}
There is only one $2$-dimensional non-abelian solvable Lie algebra, and it is of real type.  

Table \ref{s3} contains up to isomorphism all non-abelian, solvable Lie algebras of dimension $3$,  according to  \cite{ABDO}.
 After fixing an orthonormal basis $\{e_1, e_2, e_3\}$, they are described by 
  $(\ad e_1)|_\ngo  \in \glg_2(\RR)$, where $\ngo = \operatorname{span}_\RR\{e_2, e_3\}$ is an abelian ideal. 
Since $\mu^{\hg_3} \in \overline{\Gl_3(\RR)\cdot \mu^{\eg(2)}}$,
 real type is not an open condition in the space of brackets.  Here $\eg(2)$ denotes
the Lie algebra of rigid motions of the Euclidean plane and
$\hg_3$ the $3$-dimensional Heisenberg Lie algebra.

On $\Ss_3$,
the simply-connected solvable Lie group with Lie algebra  $\sg_3$,
any homogeneous Ricci flow solution
converges to the soliton on $\Ss_{3,1}$ in Cheeger-Gromov topology. This follows from Remark \ref{rmk_limitsoliton}, since the Lie brackets $\mu^{\sg_3}$
and $\mu^{\sg_{3,1}}$ lie in the same stratum, $\mu^{\sg_{3,1}}\in \overline{\Gl_3(\RR)\cdot \mu^{\sg_3}}$,
and  $S_{3,1}$ admits a solvsoliton (isometric to the hyperbolic $3$-space). 

In dimension $4$, solvable Lie algebras have rank at most $2$. The Lie algebra $\mathfrak{aff}(\CC)$ is the unique example of rank $2$ which is not of real type. From Lemma \ref{lem_realtype} and the fact that the rank is lower semi-continuous on the space of brackets, it follows that there is no sequence of brackets of real type converging to a bracket of type $\mathfrak{aff}(\CC)$. Thus, the set of solvable Lie brackets of real type is not dense in the space of all solvable Lie brackets.
\end{remark}

{\small
\begin{table}[h]
\[
\begin{array}{cccccc}
	  & (\ad{e_1})|_{\ngo} & \mbox{constraints}   & \mbox{real type} &  \mbox{flat}  \\ 
		\hline \\[-0.3cm]
	\hg_3  & \left[\begin{smallmatrix} 0&0 \\1 &0\end{smallmatrix}\right] &  - & \checkmark & - \\[0.1cm] 
	\sg_3  & \left[\begin{smallmatrix} 1& 0\\1 &1\end{smallmatrix}\right] &  - & \checkmark & - \\[0.1cm] 
	\sg_{3,\lambda} & \left[\begin{smallmatrix} 1&0 \\ 0 &\lambda\end{smallmatrix}\right] & -1\leq\lambda\leq 1 & \checkmark & - \\[0.1cm]  
	\sg_{3,\lambda}'  & \left[\begin{smallmatrix} \lambda&1\\ -1&\lambda\end{smallmatrix}\right] & 0<\lambda  & \checkmark & -
	\\[0.1cm] 
	\eg(2) & \left[\begin{smallmatrix}  0&1\\ -1& 0\end{smallmatrix}\right] & - & - & \checkmark
	\\[0.1cm] 
	\hline 
\end{array}
\]
\caption{$3$-dimensional solvable Lie algebras}\label{s3}
\end{table}}

\section{Stratification and uniqueness results for the moment map}\label{sec_stratif}

We review in this section the GIT stratification of the space of brackets $\Vs$ with respect to the linear action \eqref{eqn_Gsaction} of the real reductive Lie group $\Gs$. We also recall a uniqueness result for critical points of the moment map which will play a key role in the proof of Theorem \ref{mainthm_uniq}. The reader is referred to \cite{BL17}, \cite{GIT} for a more thorough presentation. 

Let us fix a Euclidean vector space  $(\sg, \ip)$. Denote also by $\ip$ the induced scalar products on $\glg(\sg) \simeq \sg^* \otimes \sg$ and $\Vs$.
If $\Or(\sg)$ denotes the orthogonal group of $(\sg, \ip)$, $\sog(\sg)$ its Lie algebra, and $\pg := \Sym(\sg,\ip)\subset \glg(\sg)$ the subspace of $\ip$-symmetric endomorphisms, then there is a Cartan decomposition 
\begin{equation}\label{eqn_Cartandec}
	\Gl(\sg) = \Os \exp(\pg), \qquad \glg(\sg) = \sog(\sg) \oplus \pg\,.
\end{equation}
The maximal compact subgroup $\Os$ of $\Gs$ acts orthogonally on $\Vs$ via \eqref{eqn_Gsaction}, and the elements in $\exp(\pg)$ act by symmetric maps. 
At the linear level, there is a corresponding $\glg(\sg)$-representation $\pi : \glg(\sg) \to \End(\Vs)$ given by
\begin{equation}\label{eqn_gsrep}
		\big(\pi(A) \mu \big) (\cdot, \cdot) := A \mu(\cdot,\cdot) - \mu(A\cdot, \cdot) - \mu(\cdot, A \cdot), \qquad A\in \glg(\sg), \quad \mu\in \Vs\,,
\end{equation}
which yields $\pi(\sog(\sg)) \subset \sog(\Vs)$, $\pi(\pg) \subset \Sym(\Vs,\ip)$. 
Thus, $\Gs$ is a \emph{real reductive Lie group} in the sense of \cite{GIT}, and one can study its linear action on $\Vs$ using real geometric invariant theory.

The \emph{moment map} $\mmm : \Vs \backslash \{0 \} \to \pg$ and its \emph{energy} $\normmm : \Vs\backslash \{ 0\} \to \RR$ are respectively defined by
\begin{equation}\label{eqn_defmmb}
   \la \mmm(\mu), \Alpha \ra 
    = \tfrac1{\Vert \mu\Vert^2} \cdot \la \pi(\Alpha) \mu, \mu\ra \,,  \qquad \normmm(\mu) = \Vert {\mmm(\mu)} \Vert^2,
\end{equation} 
for all $\Alpha\in \pg$, $\mu \in \Vs \backslash \{ 0\}$.
Notice that the moment map is $\Os$-equivariant: 
\begin{equation}\label{eqn_Kequivmm}
	\mmm(k \cdot \mu) = k \, \mmm(\mu) \, k^{-1}, \qquad k\in \Os, \quad \mu \in \Vs \backslash \{ 0\}.
\end{equation}
The energy $\normmm$ is therefore $\Os$-invariant. The following theorem shows how $\normmm$ determines a $\Gs$-invariant, ``Morse-type'' stratification of $\Vs \backslash \{0 \}$:

\begin{theorem}\label{thm_stratifb}\cite{GIT}
There exists a finite subset $\bca \subset \pg$ 
and a collection of smooth, $\Gs$-invariant submanifolds 
$\{ \sca_\Beta \}_{\Beta \in \bca}$ of $\Vs$, with the following properties:
\begin{itemize}
    \item[(i)]  
     We have $\Vs\backslash \{ 0\} = \bigcup_{\Beta\in \bca} \sca_\Beta$
     and $\sca_\Beta \cap \sca_{\Beta'}=\emptyset$ for $\Beta \neq \Beta'$.
    \item[(ii)] We have 
     $\overline{\sca_\Beta} \backslash \sca_\Beta \subset \bigcup_{\Beta'\in \bca, \Vert \Beta'\Vert > \Vert \Beta \Vert} \sca_{\Beta'}$ (the closure taken  in $\Vs \backslash \{0\}$).
    \item[(iii)] A bracket $\mu$ is contained in
       $\sca_\Beta$ if and only if the negative gradient 
             flow of $\normmm$ starting at $\mu$ 
             converges to a critical point $\mu_C$ of $\normmm$
             with $\mmm(\mu_C) \in \Os \cdot \Beta$. 
\end{itemize} 
\end{theorem}



We now describe the strata in more detail. For $\Beta\in \pg$ we set 
\[
	\Beta^+ := \Beta + \Vert \Beta \Vert^2 \Id_\sg\,.
	\]
Denote by $\Vr \subset \Vs$ the eigenspace of $\pi(\Beta^+) = \pi(\Beta) - \Vert \Beta \Vert^2 \, \Id_{\Vs}$ corresponding to the eigenvalue $r\in \RR$ (recall that $\pi(\Id_\sg) = - \Id_\Vs$), and consider the subspace
\begin{equation}\label{eqn_defVnn}
   \Vnn := \bigoplus_{r \geq 0} \Vr\,.
\end{equation}
There exist subgroups of $\Gs$ adapted to these subspaces. In order to describe them, since 
the linear map $\ad(\Beta) : \glgs \to \glgs$, $A \mapsto [\beta, A]$ is  self-adjoint, 
we may decompose $\glg(\sg) = \bigoplus_{r\in \RR} \glg(\sg)_r$ as a sum of $\ad(\Beta)$-eigenspaces, and set accordingly
\begin{equation}\label{eqn_guqbeta}
    \ggo_\Beta := \glg(\sg)_0 = \ker (\ad(\Beta) ), \qquad \ug_\Beta := \bigoplus_{r> 0} \glg(\sg)_r, \qquad \qg_\Beta := \ggo_\Beta \oplus \ug_\Beta.
\end{equation}
We then denote by 
\[
    \Gb := \{ g \in \Gs : g \Beta g^{-1} = \Beta \}, \qquad \Ub := \exp(\ug_\Beta), \qquad \Qb := \Gb \Ub,
\] 
the centralizer of $\Beta$ in $\Gs$, the unipotent subgroup associated with $\Beta$, and the parabolic subgroup associated with $\Beta$, respectively. Set $\Kb := \Os \cap \Gb$ and consider also
\[
   \Hb := \Kb \, \exp(\pg \cap \hg_\Beta), \qquad  \hg_\Beta := \{ \Alpha \in \ggo_\Beta : \langle \Alpha,  \Beta\rangle = 0\}, \qquad
\]
a codimension-one reductive subgroup (resp.~ subalgebra) of $\Gb$ (resp.~ $\ggo_\Beta$).

The groups $\Gb$, $\Ub$ and $\Qb$ are closed in $\Gs$, and $\Ub$ is normal in $\Qb$. They satisfy  $\Gb \cdot \Vr \subset \Vr$ for all $r$, and $\Qb \cdot \Vnn \subset \Vnn$. The subgroup $\Gb$ is real reductive, with Cartan decomposition $\Gb = \Kb \exp(\pg_\Beta)$, $\pg_\Beta = \pg \cap \ggo_\Beta$, induced from that of $\Gs$. The same holds for $\Hb$, and in fact $\Gb = \exp(\RR \Beta) \times \Hb$ is a direct product. A key property of $\Qb$ is that it is large enough so that we have
\begin{equation}\label{eqn_GOQ}
	\Gs = \Os \Qb.
\end{equation}
For a critical point $\mu_C$ of $\normmm$ we set $\Beta := \mmm(\mu_C)$ and define
\[
 \Vzeross := \big\{ \mu \in \Vzero : 0\notin \overline{\Hb\cdot \mu} \big\}\,, \qquad  \Vnnss := p_\Beta^{-1} \big(\Vzeross\big),
\]
where $p_\Beta : \Vnn \to \Vzero$ denotes the orthogonal projection. The set $\Vzeross$ (resp.~ $\Vnnss$) is open and dense in $\Vzero$ (resp.~ in $\Vnn$).
The proof of Theorem \ref{thm_stratifb} in fact shows that
\begin{equation}
	\sca_\Beta =\Os \cdot \Vnnss. \label{eqn_SbetaKU}
\end{equation}
In particular, for any $\mu \in \sca_\Beta$ one can always find $k\in \Os$ such that $k \cdot \mu \in \Vnn$. 

\begin{definition}\label{def_gauged}
We say a bracket $\mu \in \Sb \subset \Vs$ is \emph{gauged correctly} 
w.r.t. $\Beta$, if $\mu \in \Vnnss$.
\end{definition}

Later on, we will fix a stratum label $\beta$, and then exploit \eqref{eqn_SbetaKU} to gauge everything, in order to work on the set $\Vnnss$ which is better adapted to $\Beta$.

Finally, we recall the following uniqueness result for critical points of $\normmm$. We think of it as a generalization of the Kempf-Ness theorem, which gives `uniqueness' of minimal vectors (zeros 
of the moment map).

\begin{theorem}\label{thm_GITuniq}\cite{GIT}
For $\mu \in \sca_\Beta$, the set of critical points of $\normmm$ contained in $\overline{\Gs \cdot \mu} \cap \sca_\Beta$ equals $\RR_{>0} \cdot \Os \cdot \mu_C$, where 
$\mu_C \in \sca_\Beta$ is a critical point of $\normmm$ with $\mmm(\mu_C) = \Beta$.
\end{theorem}

\section{Uniqueness of solvsolitons}\label{sec_uniq}

The main result of this section is Theorem \ref{thm_uniqsolvsol}, a `solvsoliton analogue' of Theorem \ref{thm_GITuniq}. Before turning to that, we recall  the correspondence between left-invariant metrics on
a Lie group and Lie brackets.

Let $\Ss$ be a simply-connected Lie group whose Lie algebra $\sg$ is endowed with a fixed background scalar product $\ip$. Denote by $\mu^\sg \in \Vs$ the Lie bracket of $\sg$. The set $\mca_{\mathsf{left}}(\Ss)$ of left-invariant Riemannian metrics  on $\Ss$ can be parameterized by the orbit $\Gs \cdot \mu^\sg  \subset \Vs$ as follows: any  $g' \in \mca_{\mathsf{left}}(\Ss)$ is determined by a scalar product $\ip'$ on $\sg$, which may be written as $ \ip' = \la h \, \cdot, h \, \cdot\ra$ for some $h\in \Gs$. We then associate to $g'$ the bracket $h \cdot \mu^\sg$. Recall that $\Gl(\sg)$ acts on $\Vs$ via \eqref{eqn_Gsaction}. 

Conversely, to a bracket $h \cdot \mu^\sg \in \Gs \cdot \mu^\sg$ we associate the left-invariant metric on $\Ss$ determined by the scalar product $\la h \,\cdot , h\,\cdot\ra$ on $\sg$. Notice that in both directions the map $h$ is not unique, thus this correspondence is one to one only when we take into account the action of the groups $\Aut(\sg,\mu^\sg) \simeq \Aut(\Ss)$ and $\Os$:
\begin{equation}\label{eqn_metricsbrackets}
	\mca_{\mathsf{left}}(\Ss) \,  /  \, \Aut(\Ss)  \quad  \leftrightsquigarrow \quad \Gs \cdot \mu^\sg \, / \, \Os.
\end{equation}
Here $\Aut(\Ss)$ acts by pull-back on $\mca_{\mathsf{left}}(\Ss)$, each orbit consisting of pairwise isometric metrics.
To every Lie bracket $\mu \in \Vs$ there corresponds a Riemannian manifold $(\Ss_\mu, g_\mu)$, where $\Ss_\mu$ is the simply-connected Lie group with Lie algebra $(\sg,\mu)$, and the metric
$g_\mu \in \mca_{\mathsf{left}}(\Ss_\mu)$ is determined by $\ip$ on $\sg \simeq T_ e \Ss_\mu$.

\begin{definition}
A \emph{solvsoliton} $(\Ss, g_{\mathsf{sol}})$ is a solvmanifold for which
\begin{equation}\label{eqn_solvsoliton}
      \Ricci_{g_{\mathsf{sol}}} = c \cdot \Id_\sg + D, \qquad c\in \RR, \qquad D\in \Der(\sg),
\end{equation}
 where  $\Ricci_{g_{\mathsf{sol}}}$ denotes the Ricci endomorphism at $e\in \Ss$ and $\Der(\sg)$ is the Lie algebra of derivations of $\sg$.  
The corresponding Lie bracket $ \mu \in \Gs \cdot \mu^\sg$ is also called a solvsoliton.
\end{definition}


We now turn to the main result of this section.

\begin{theorem}\label{thm_uniqsolvsol}
Let $(\sg,\mu)$ be a solvable Lie algebra of real type with $\mu \in \sca_\Beta$. Then, up to scaling
the set of solvsolitons in $\overline{ \Gl(\sg) \cdot \mu}\cap \sca_\Beta$
 is contained  in a unique $\Or(\sg)$-orbit.
\end{theorem}

This immediately implies

\begin{corollary}\label{cor_uniqsolvsolmetric}
Let $(\Ss, \gsol)$ be a non-flat solvsoliton. Then, any other left-invariant Ricci soliton metric on $\Ss$ is of the form $\alpha \cdot \psi^* \gsol$, for some $\alpha > 0$, $\psi\in \Aut(\Ss)$.
\end{corollary}

\begin{proof}
The group $\Ss$ is of real type by Remark \ref{rem_solvsolreal}. Hence by \cite[Prop.~8.4]{Jbl2015} any left-invariant Ricci soliton $g$ on $\Ss$ is a solvsoliton.
After rescaling $g$, by Theorem \ref{thm_uniqsolvsol} we may assume that $\gsol$ and $g$ have associated solvsoliton brackets $\musol, \mu \in V(\sg)$  with 
 $\musol = k\cdot \mu$ for some $k\in \Or(\sg)$. 
 Thus, by \eqref{eqn_metricsbrackets} the metrics $\gsol$ and $g$ are isometric via an automorphism.
\end{proof}

The  uniqueness of solvsolitons up to equivariant isometry was known for completeley solvable groups: see \cite{Heber1998} for the Einstein case and \cite{solvsolitons} for solvsolitons.

We now work towards a proof of Theorem \ref{thm_uniqsolvsol}. The idea is that starting with a solvsoliton bracket one can explicitly construct a bracket on the same $\Gl(\sg)$-orbit which is a critical point of the energy map $\normmm$, and then Theorem \ref{thm_GITuniq} can be applied. 


Let us recall the following formula for the Ricci endomorphism $\Ricci_\mu \in \Syms$ of 
a Lie bracket $\mu \in V(\sg)$:
\begin{equation}\label{eqn_Ricmu}
    \Ricci_\mu \, = \, \mm_\mu - \unm \kf_\mu - \unm \left(\ad_\mu \mcv_\mu  + (\ad_\mu \mcv_\mu) ^t \right)\,.
\end{equation}
Here,  $\mm_\mu = \unc \cdot \mmm(\mu) \cdot \Vert \mu \Vert^2$ is a multiple of the moment map $\mmm(\mu)$ defined in \eqref{eqn_defmmb} and
$\la \kf_\mu X, Y \ra = \tr (\ad_\mu X \ad_\mu Y)$ is the endomorphism
associated to the Killing form. The mean curvature vector $\mcv_\mu$ is implicitly defined by $\la \mcv_\mu, X \ra = \tr \ad_\mu X$ for all $X\in\sg$, and $(\cdot)^t$ denotes the transpose with respect to $\ip$. The \emph{modified Ricci curvature} is defined by
\begin{equation}\label{eqn_Ricmod}
  \Riccim_\mu \, := \, \mm_\mu - \unm \kf_\mu.
\end{equation}
Moreover, we set $\scalm(\mu)=\tr \Riccim_\mu$. Notice, that
for non-flat solvmanifolds $\scalm(\mu)<0$ by  Lemmas 3.5 and 3.6 in \cite{BL17}.

In terms of the stratification from Section \ref{sec_stratif}, by \cite[Thm.~6.4 $\&$ Cor.~C.3]{BL17} we have

\begin{proposition}\cite{BL17}\label{prop_solvsol}
Let $\musol \! \in \Vnnss \subset \! \sca_\Beta$ be a solvsoliton with  $\scalm(\musol) = -1$. Then,
\[
      \Riccim_\musol \, = \, \Beta \, =  \, c \cdot \Id_\sg + D,  \qquad \, c =  - \Vert \Beta \Vert^2, \quad D =  \Beta^+ \in \Der(\sg, \musol).
\]
Moreover, $\Beta^+ = \Beta + \Vert \Beta \Vert^2 \cdot \Id_\sg  \geq 0$, and its image is the nilradical of $(\sg,\musol)$.
\end{proposition}


Next, let us briefly review some of the structural results for solvsolitons from \cite{solvsolitons}. Given a solvsoliton bracket $\mu\in \Vs$ with nilradical $\ngo$, consider the orthogonal decomposition $\sg = \ag \oplus\ngo$. We have that $\mu(\ag,\ag) = 0$, and for all $Y\in \ag$, $X\in \ngo$, it holds that
\begin{equation}\label{eqn_properties}
   \left[\ad_\mu Y, (\ad_\mu Y)^t\right] = 0, \qquad \tr \left((\ad_\mu Y) \, (\ad_\mu X)^t \right) = 0\,.
\end{equation}
Moreover, the symmetric endomorphism $\mm_\mu$ defined after equation \eqref{eqn_Ricmu} satisfies 
\begin{equation}\label{eqn_mmmuan}
  \la \mm_\mu Y, Y \ra  = -\unm  \Vert {\ad_\mu Y} \Vert^2, \qquad \la \mm_\mu Y, X\ra = 0, \qquad \la \mm_\mu X, X \ra= \la \mm_\nu X, X\ra,
\end{equation}
for all $Y\in \ag$, $X\in \ngo$. Here $\nu : \ngo\wedge \ngo \to \ngo$ denotes the restriction of $\mu$ to $\ngo$ (see \cite[Prop. 4.13]{LafuenteLauret2014b}).
This yields in particular $\mm_\mu|_\ag < 0$, since $\ad_\mu Y = 0$ implies $Y\in \ngo$.
 
The next lemma shows how to modify a solvsoliton  to obtain a critical point of  $\normmm$.

\begin{lemma}\label{lem_solvsolmmsol}
Let $\mu\in \sca_\Beta$ be a solvsoliton Lie bracket with $\Ricci_\mu^* = \Beta$, and let $\sg = \ag \oplus \ngo$ be the orthogonal decomposition where $\ngo$ is the nilradical of $\mu$.
\begin{itemize}
  \item[(i)] If $h = \minimatrix{h_\ag}{0}{0}{\Id_\ngo} \in \Gl(\sg)$, then
  \[
    \mm_{h\cdot \mu} = (h^{-1})^t \, \mm_\mu \, h^{-1}, \qquad \Riccim_{h\cdot \mu} = (h^{-1})^t \, \Riccim_\mu \, h^{-1}.
  \]
  \item[(ii)] There exists $h = \minimatrix{h_\ag}{0}{0}{\Id_\ngo} \in \Gl(\sg)$ such that $h \cdot \mu$ is a critical point of the norm squared of the moment map $\normmm$, with $\mmm(h\cdot \mu) = \Beta$.
\end{itemize}
\end{lemma}

\begin{proof}
Set $\mub := h \cdot \mu$. Notice that $\mub(\ag, \ag) = 0$ and $\mub|_{\ngo \wedge \ngo} = \nu$. By \cite[Lemma 4.4]{LafuenteLauret2014b}, if $\{ Y_i\}$ is an orthonormal basis of $\ag$ then for $Y\in \ag$, $X\in \ngo$ we have that 
\begin{align*}
	\langle \mm_\mub Y, Y \rangle &= -\unm \tr \ad_\mub Y (\ad_\mub Y)^t, \\
	\langle \mm_\mub X, X \rangle &= \langle \mm_\nu X, X\rangle + \unm \sum \big\langle [\ad_\mub Y_i, (\ad_\mub Y_i)^t] X, X \big\rangle, \\
	\langle \mm_\mub Y, X \rangle &= -\unm \tr \ad_\mub Y (\ad_\mub X)^t.
\end{align*}
Since $(\ad_\mu Y)(\ngo) \subset \ngo$ and $\ad_\mu Y|_\ag = 0$, \eqref{eqn_adconj} implies that  $\ad_{\mub} Y = \ad_\mu (h^{-1} Y)$ for any $Y\in \ag$. Using that and \eqref{eqn_properties}, \eqref{eqn_mmmuan} one can easily verify the formula for $\mm_\mub$.
Since $\kf_{\mub} = (h^{-1})^t \kf_\mu h^{-1}$ (see \cite[Lemma 3.7]{homRF}), the formula for the modified Ricci curvature also follows.

To prove (ii), we look for a map $h = \minimatrix{h_\ag}{0}{0}{\Id_\ngo} \in \Gl(\sg)$ such that $\mm_{h\cdot \mu} = \Beta$. It will then follow that $h\cdot \mu$ is a critical point of the energy map $\normmm$. Indeed, $\mm_{h\cdot \mu}=\tfrac{1}{4}\mmm(h\cdot \mu)\cdot \Vert h\cdot \mu\Vert^2$,  and $\tr \mmm(h\cdot \mu) = -1 = \tr\Beta$  by \cite[Lemma 3.7]{homRF}, 
from which we deduce $\mmm(h\cdot \mu) = \Beta$. But $\normmm \geq \Vert \beta \Vert^2$ on $\sca_\Beta$ by Theorem \ref{thm_stratifb}, (iii).  

To that end, let $h = \minimatrix{h_\ag}{0}{0}{\Id_\ngo} \in \Gl(\sg)$ satisfying
\[
  h^t h =  \Id_\sg - \tfrac{1}{2 \Vert \Beta \Vert^2} \cdot \kf_\mu.
\]
Such an $h$ exists if and only if the right-hand-side is positive definite. But $\kf_\mu |_\ngo = 0$, thus on $\ngo$ we have $\Id_\ngo$. And the restriction to $\ag$ equals $- \Vert \Beta\Vert^{-2} \mm_\mu|_\ag > 0$, by the fact that $\Riccim_\mu = \Beta$ and Proposition \ref{prop_solvsol}. Hence there is at least one such $h$.
Using (i) we obtain
\begin{align*}
  \mm_{h\cdot \mu} =& \, \,(h^{-1})^t \, \mm_\mu \, h^{-1} = (h^{-1})^t \, \big( \Riccim_\mu + \unm \cdot \kf_\mu \big) \, h^{-1} \\
  =&  \, \, (h^{-1})^t \, \big( \Beta^+ - \Vert \Beta \Vert^2 \cdot \Id_\sg  + \unm \cdot \kf_\mu \big) \, h^{-1} \\
  =& \, \,  (h^{-1})^t \, \big( \Beta^+ - \Vert \Beta \Vert^2 \cdot h^t h \big) \, h^{-1} = \Beta,
\end{align*}
where in the last step we are using the identity $(h^{-1})^t \, \Beta^+ \, h^{-1} = \Beta^+$, which follows at once from the special form of $h$ and the properties of $\Beta^+$ stated in Proposition \ref{prop_solvsol}. 
\end{proof}

\begin{proof}[Proof of Theorem \ref{thm_uniqsolvsol}]
Let $\mu_1, \mu_2 \in \overline{\Gl(\sg) \cdot \mu^\sg} \cap \sca_\Beta$ be two solvsoliton brackets. By \eqref{eqn_SbetaKU} and Proposition \ref{prop_solvsol}, after acting with $\Or(\sg)$, we may assume that 
 $\Riccim_{\mu_1} = \Riccim_{\mu_2} = \Beta$
and that the nilradicals of $\mu_1$ and $\mu_2$ equal to $\ngo$. Setting $\sg = \ag\oplus \ngo$,
by Lemma \ref{lem_solvsolmmsol}, (ii)
there exist maps $h_i = \minimatrix{(h_i)_\ag}{0}{0}{\Id_\ngo} \in \Gs$, $i=1,2$, such that $h_i\cdot \mu_i \in \overline{\Gl(\sg) \cdot \mu^\sg}$ are critical points of $\normmm$ with $\mmm(h_i \cdot \mu_i) = \Beta$. Theorem \ref{thm_GITuniq} then yields $h_1\cdot \mu_2 = (kh_2) \cdot \mu_2$ for some $k\in \Or(\sg)$. This implies that $\Beta = \mmm(h_1 \cdot \mu_1) = k \mmm(h_2 \cdot \mu_2) \, k^{-1} = k \, \Beta \, k^{-1}$ by \eqref{eqn_Kequivmm}, and hence $k$ commutes with $\Beta$ and $\Beta^+$, thus 
$k = \minimatrix{k_\ag}{0}{0}{k_\ngo}$. After acting on $\mu_2$ with $\minimatrix{\Id_\ag}{0}{0}{k_\ngo^{-1}}$ we may assume $k_\ngo = \Id_\ngo$. Hence, 
 $ \mu_1 = h \cdot \mu_2$ for $ h = \minimatrix{h_\ag}{0}{0}{\Id_\ngo} \in \Gl(\sg)$.
Finally, using $\Riccim_{\mu_1} = \Riccim_{\mu_2} = \Beta$, 
the fact that $\Beta|_\ag = c \cdot \Id_\ag$, see Proposition \ref{prop_solvsol}, and Lemma \ref{lem_solvsolmmsol}, (i), 
we get $h_\ag^t h_\ag = \Id_\ag$, from which it follows that $h\in \Or(\sg)$. 
\end{proof}

\section{Proof of Theorem \ref{mainthm_uniq}}\label{sec_thmA}

Before turning to the proof of Theorem \ref{mainthm_uniq}, we discuss here one of our main tools, an ODE on the space of brackets which is equivalent to the Ricci flow of left-invariant metrics.

It was shown in \cite{homRF}  that each Ricci flow solution of left-invariant metrics on $\Ss$ corresponds to a curve of brackets in $\Vs$ solving the \emph{bracket flow} $\mu'=-\pi(\Ricci_\mu)\mu$. Since the right-hand-side is always tangent to the $\Gs$-orbit, the flow preserves orbits. Hence if the initial bracket $\mu(0)$ belongs to a stratum $\Sb$, then also $\mu(t) \in \Sb$ for all $t$ for which the solution exists, since $\Sb$ is $\Gs$-invariant.
 Moreover, the flow can be `gauged' so that it consists of brackets which are gauged correctly w.r.t.~$\Beta$ (Def.~ \ref{def_gauged}). To that end, consider the subspace $\kg_{\ug_\Beta} = \{A - A^t : A\in \ug_\Beta \} \subset \sog(\sg)$, a direct complement for $\qg_\Beta$: 
\begin{equation}\label{eqn_qku}
	\glg(\sg) = \qg_\Beta \oplus \kg_{\ug_\Beta}\,.
\end{equation}
Denote by $(\cdot)_{\qg_\Beta}$ the corresponding linear projection onto $\qg_\Beta$.

\begin{remark}\label{rmk_projqbeta}
In general \eqref{eqn_qku} is not an orthogonal decomposition. We do have an orthogonal decomposition $\glgs = \ug_\Beta \oplus \ggo_\Beta \oplus \ug_\Beta^t$, and for $A = A_{\ug_\Beta} + A_{\ggo_\Beta} + A_{\ug_\Beta^t}\in \glgs$ 
 the projection according to \eqref{eqn_qku} is given by $A_{\qg_\Beta} = A_{\ggo_\Beta} + A_{\ug_\Beta}+(A_{\ug_\Beta^t} )^t$. In particular, for $A \in \Syms$
we have  $A_{\qg_\Beta} = A_{\ggo_\Beta} + 2\cdot A_{\ug_\Beta}$
and
 \[
   \Vert A \Vert \leq   \Vert A_{\qg_\Beta}\Vert \leq 2 \cdot \Vert A \Vert\,.
 \]
\end{remark}

Given a Ricci flow solution $(g(t))\subset \mca_{\mathsf{left}}(\Ss)$ with $g(0) = g_0$, we write $g_0$ at the point $e \in \Ss$ with respect to the fixed scalar product $\ip$ on $\sg \simeq T_e \Ss$, 
that is $(g_0)_e(\cdot,\cdot) = \la h_0 \, \cdot, h_0 \, \cdot \ra$, and we may assume that $h_0 \in \Qb$ by \eqref{eqn_GOQ}. We call the following ODE the \emph{gauged bracket flow}:
\begin{equation}\label{eqn_QbgaugedBF}
	\frac{\rm d  \mu}{ {\rm d} t} = - \pi \big( ( \Riccim_{\mu} )_{\qg_\Beta}  \big) \, \mu, \qquad \mu(0) = h_0\cdot \mu^\sg \in \Qb\cdot \mu^\sg .
\end{equation}

\begin{theorem}\cite{homRF,BL17}\label{thm_BFRFequiv}
Let $(\Ss, g_0)$ be a non-abelian, simply-connected Lie group with left-invariant metric $g_0$ and Lie algebra $(\sg, \mu^\sg)$, and consider a fixed scalar product $\ip$ on $\sg$. Then, the solution $(g(t))_{t \in [0,\infty)}$ to the Ricci flow starting at $g_0$ and the solution $(\mu(t))_{t \in [0,\infty)}$ to the gauged bracket flow \eqref{eqn_QbgaugedBF} starting at $h_0 \cdot \mu^\sg$
differ only by pull-back by time-dependent diffeomorphisms. Here, $h_0\in \Qb$ is such that $(g_0)_e(\cdot,\cdot) = \la h_0\, \cdot, h_0\,\cdot \ra$.
\end{theorem}

The proof shows the existence of Lie group isomorphisms $\varphi_t : \Ss \to \Ss_{\mu(t)}$ such that $g(t) = \varphi_t^* \,  g_{\mu(t)}$. They can be obtained from the corresponding Lie algebra isomorphisms $h(t) : (\sg, \mu^\sg) \to (\sg, \mu(t))$, which satisfy the ODE
\begin{equation}\label{eqn_ODEh}
	h' =  - \big( {\Riccim_{\mu(t)} } \big)_{\qg_\Beta}  \cdot  h ,\qquad h(0) = h_0.
\end{equation}
The gauging is chosen so that for $\mu(0) \in \Vnnss$ we have $\mu(t) \in \Qb \cdot \mu^\sg \subset \Vnnss$ for all $t \in [0,\infty)$.

\begin{remark}\label{rmk_gbfequiv}
Unlike the bracket flow, the gauged bracket flow \eqref{eqn_QbgaugedBF} is not $\Os$-equivariant. However, it is still $\Kb$-equivariant: for $k\in \Kb$ one has that
\begin{align*}
	k \cdot \left(  -\pi \big( ( \Riccim_{\mu} )_{\qg_\Beta}  \big) \mu \right) =& -\pi \big( k ( \Riccim_{\mu} )_{\qg_\Beta}  k^{-1} \big) (k \cdot \mu) = - \pi \big(  (k \Riccim_\mu k^{-1})_{\qg_\Beta} \big) (k\cdot \mu) \\
	=&  -\pi \big(  (\Riccim_{k\cdot\mu})_{\qg_\Beta}\big) (k\cdot \mu).
\end{align*}
The second identity follows from Remark \ref{rmk_projqbeta}, since conjugation by $k\in \Kb$ preserves  $\ggo_\Beta$ and $\ug_\Beta$, thus for $A\in \Sym(\sg)$ we have that $k A_{\qg_\Beta} k^{-1} = (k A k_{-1})_{\qg_\Beta}$.
\end{remark}

Next, we  recall a scale-invariant Lyapunov function, 
which is monotone along immortal solutions to \eqref{eqn_QbgaugedBF}: 
see \cite[$\S$7]{BL17}.
Consider the following codimension-one subgroup of $\Qb$,
\[
    \Slb := \Hb \Ub \subset \Qb,
\]
and let $\slgb$ be its Lie algebra. Assume without loss of generality that $\mu^\sg \in \Vnnss$.  Since 
 $\Qb \cdot \mu^\sg$ is a cone over  $\Slb \cdot \mu^\sg$
by \cite{BL17}, for every $\mu \in \Qb \cdot \mu^\sg$ there exists a unique scalar $v_\Beta(\mu) \in \RR_{>0}$ such that
\[
    v_\Beta (\mu)  \, \mu \in \Slb \cdot \mu^\sg.
\] 
We call $v_\Beta$ the \emph{$\Beta$-volume functional}; notice that it depends on the `base bracket' $\mu^\sg$. It has the property that for some constant $C_{\mu^\sg} >0$ and for all $\mu \in \Qb\cdot \mu^\sg$ we have 
\begin{equation}\label{eqn_lowbdvbeta}
  v_\Beta(\mu) \geq C_{\mu^\sg} \cdot \Vert \mu\Vert^{-1}.
\end{equation}

\begin{theorem}\cite{BL17}\label{thm_lyapunov}
Let $(\mu(t))_{t\in [0,\infty)}$ be a solution to \eqref{eqn_QbgaugedBF} with $\mu(0) \in \Qb\cdot \mu^\sg $. Then,  
\[
    F_\beta : \Qb \cdot \mu^\sg \to \RR,  \qquad \mu \mapsto v_\Beta(\mu)^2 \cdot \scalm(\mu),
\]
is scale-invariant and evolves along \eqref{eqn_QbgaugedBF} by
\begin{equation}\label{eqn_Fbetanondec}
    \ddtbig \,  F_\Beta( \mu)\, = \, 
     2 \cdot v_\Beta(\mu)^2 \cdot    \Big( \left\Vert {\Riccim_{\mu} } \right\Vert^2   + \scalm({\mu}) \cdot  \langle \Riccim_\mu, \Beta \rangle \Big) \geq 0\,.
\end{equation}
Equality holds for some $t>0$ if and only if $\mu(0)$ is a solvsoliton. 
\end{theorem}

The monotonicity of $F_\Beta$ follows from Cauchy-Schwarz and the estimate
\begin{equation}\label{eqn_ricbetapgeq0}
	\la \Riccim_\mu , \Beta \ra \geq \vert \scalm_\mu \vert \cdot \Vert \Beta\Vert^2,
\end{equation}
which holds on $\Qb\cdot \mu^\sg$. Moreover, even though $F_\Beta$ is defined only on one orbit, the rigidity statement also holds for all potential limits:

\begin{proposition}\cite{BL17}\label{prop_rigidity}
Let $\bar \mu \in \overline{\Qb \cdot \mu^\sg}$ with $\scalm(\bar \mu) = -1$. Then,  
\[
 \left\Vert {\Riccim_{\bar\mu} } \right\Vert^2   - \langle \Riccim_\mu, \Beta\rangle = 0
\]
  if and only if $\bar \mu$ is a solvsoliton with $\Riccim_{\bar \mu} = \Beta$ and $\bar \mu \in \sca_\Beta$.
\end{proposition}

We are now in a position to prove the main result of the article.

\begin{proof}[Proof of Theorem \ref{mainthm_uniq}]
Let $(\Ss,g_0)$ be a solvmanifold of real type, with Lie algebra $(\sg, \mu^\sg)$. By Theorem \ref{thm_BFRFequiv},  since $\mu^\sg \neq 0$, to the Ricci flow solution $g(t)$ with $g(0) = g_0$ there corresponds a solution $\mu(t)$ to the gauged bracket flow \eqref{eqn_QbgaugedBF}.  Recall that $\mu(t) \in \Qb \cdot \mu^\sg$ for all $t\geq 0$, and that $\Qb\cdot \mu^\sg \subset \sca_\Beta$ for some $\Beta$.

For non-compact homogeneous spaces there exists a normalized bracket 
flow keeping the modified scalar curvature $\scalm$ constant. More precisely,  since  $\scalm(\mu(t)) < 0$ for all $t\geq 0$, as explained in \cite[$\S$3.3]{homRF}, after an appropriate time reparameterization the corresponding $\scalm$-normalized family $\nu(t) := \vert {\scalm ({\mu(t)}) } \vert^{-1/2} \cdot \mu(t)$ solves
\begin{equation}\label{eqn_normalizedgaugedBF}
  \frac{\rm d  \nu}{ {\rm d} t} = - \pi \Big( ( \Riccim_{\nu} )_{\qg_\Beta}  + \Vert{ \Riccim_\nu}\Vert^2 \cdot \Id_\sg \Big) \, { \nu}, \qquad { \nu}(0) = \nu_0 \in \Qb\cdot \mu^\sg \,.
\end{equation}
To this end, recall that by \cite{homRF} 
 we have ${\rm d} \scalm |_\mu (\pi(A)\mu) = - 2 \cdot \la\Riccim_\mu, A\ra$.
 Thus $\scalm(\nu(t)) \equiv -1$,  since
$ ( \Riccim_{\nu} )_{\qg_\Beta}= \Riccim_{\nu} - (\Riccim_{\nu} )_{ \kg_{\ug_\Beta}}$
and $\Riccim_\nu  \perp ( \Riccim_{\nu} )_{ \kg_{\ug_\Beta}}$, see \eqref{eqn_qku}.

We first show that there exist $c_{\mu_0}, C_{\mu_0} > 0$  such that for all $t\geq 0$ it holds
\begin{equation}\label{eqn_nubounded}
	0 < c_{\mu_0} \leq \Vert \nu(t) \Vert \leq C_{\mu_0}\,.
\end{equation}
The existence of $c_{\mu_0}$ is clear since $\scalm(0)=0$. 
On the other hand, if some subsequence $\nu(t_k)$ is unbounded, then 
the sequence $\tilde \nu_k := \nu(t_k) / \Vert \nu(t_k) \Vert$ satisfies
 $\scalm(\tilde \nu_k) \to 0$ as $k\to\infty$. A subsequential limit would then contradict the real type hypothesis: see Lemma \ref{lem_realnonflat}. 

Next, we claim that the $\omega$-limit of the solution $(\nu(t))_{t\in [0,\infty)}$ consists entirely of solvsoliton brackets lying in the same stratum $\sca_\Beta$.  This will follow from Proposition \ref{prop_rigidity}, once we show that 
the non-negative function 
$f(t):= \Vert {\Riccim_{\nu(t)} }\Vert^2 - \langle \Riccim_{\nu(t)}, \Beta\rangle$ tends to $0$ as $t\to \infty$. To see that, notice first that by scale-invariance of the Lyapunov function $F_\Beta$  (Theorem \ref{thm_lyapunov}) 
we have that $(\grad F_\Beta)_\nu \perp \RR \,\nu$.
Hence,
 along the normalized bracket flow \eqref{eqn_normalizedgaugedBF} $F_\Beta$ satisfies the same evolution equation \eqref{eqn_Fbetanondec}. Together with \eqref{eqn_nubounded} and \eqref{eqn_lowbdvbeta} this implies
\[
	\ddt F_\Beta(\nu) \geq C'_{\mu_0} \cdot f(t) \geq 0,
\]
for some constant $C'_{\mu_0} >0$. Since $F_\Beta$ is monotone non-decreasing and 
$F_\Beta(\nu(t)) < 0$ for all $t\geq 0$, it follows that
$\int_0^\infty f(t) \,   dt  <\infty$. 
On the other hand, notice that with respect to a fixed orthonormal basis of $V(\sg)$
the entries of $\ddt \nu$ are polynomials in the entries of $\nu$. By using again the upper bound in \eqref{eqn_nubounded} we deduce that $f'(t)\leq D_{\mu_0}$ 
for all $t \geq 0$. It is now clear that $ \lim_{t\to\infty} f(t) = 0$, which proves our claim.

Applying Theorem \ref{thm_uniqsolvsol} we conclude that the $\omega$-limit is contained in a single $\Os$-orbit of non-flat solvsolitons $\Os \cdot \nu_{\mathsf{sol}}$. From this we deduce that it is equivalent to normalize the scalar curvature. Indeed, by $\Os$-invariance of $\scal$, we have
$0 > s_\infty = \lim_{t\to \infty} \scal(\nu(t)) $. Hence, the $\omega$-limit of the  $\scal$-normalized bracket family 
\begin{eqnarray}
  \big(\vert \scal(\mu(t))\vert^{-1/2} \, \mu(t) \big)_{t\in [0,\infty)}=
   \big(\vert \scal(\nu(t)) \vert^{-1/2} \, \nu(t) \big)_{t\in[0,\infty)}\label{eqn_normconvergence}
 \end{eqnarray}
 is contained in $\Or(\sg) \cdot  \tilde\nu_{\mathsf{sol}}$, where $ \tilde \nu_{\mathsf{sol}} := \vert  s_\infty \vert^{-1/2} \, \nu_{\mathsf{sol}}$.

The bracket $ \tilde \nu_{\mathsf{sol}}$ corresponds to a solvmanifold $(\bar \Ss, \bgsol)$, 
 which by Theorem \ref{thm_uniqsolvsol} does not depend (up to isometry) on the initial metric $g_0$. By \cite[Corollary 6.20]{Lauret2012}, bracket convergence implies Cheeger-Gromov subconvergence to a space locally isometric to 
$(\bar \Ss, \bgsol)$.
 If we set $\bar g(t) := \vert \scal(t)\vert \cdot g(t)$, this says that any sequence $(\Ss, \bar g(t_k))_{k\in\NN}$, $t_k\to\infty$, has a subsequence converging in Cheeger-Gromov topology to a Riemannian manifold locally isometric to $(\bar \Ss, \bgsol)$. 
 By Theorem D.2 in \cite{BL17}, the limit is in fact simply-connected, and hence equal to $(\bar\Ss,\bgsol)$, as claimed.
\end{proof}

\begin{remark}\label{rmk_limitsoliton}
The above proof shows that in fact the $\scal$-normalized Ricci flow converges to the unique soliton whose bracket lies in $\overline{\Gl(\sg)\cdot \mu} \cap \sca_\Beta$, see Theorem \ref{thm_uniqsolvsol}.
\end{remark}

\section{No algebraic collapsing}\label{sec_nocollapse}

Let $(M^n,g_k)_{k\in \NN}$ be a sequence of Riemannian metrics converging in 
pointed Cheeger-Gromov topology to a limit space $\big(\bar M^n, \bar g\big)$. Assume that for each $k \in \NN$ there is a  connected, $n$-dimensional Lie group $\G_k$ of $g_k$-isometries acting transitively on $M^n$. A natural way of obtaining an isometric group action in the limit is by arguing at the infinitesimal level, as follows: for each $k\in \NN$ one considers $n$ linearly independent $g_k$-Killing fields, which after a suitable normalization subconverge in $C^1$-topology to $n$ linearly independent $\bar g$-Killing fields on $\bar M^n$; see \cite[$\S$6.2]{Heber1998} and \cite[$\S$9]{BL17}. The sequence $(M^n, g_k)_{k\in \NN}$ is called \emph{algebraically non-collapsed}, if the $n$ limit Killing fields span the tangent space at each point of $\bar M^n$. Notice that if this is the case, then after lifting them to the universal cover $(X^n, \bar g)$ of $(\bar M^n,\bar g)$, they can be `integrated' to a simply-transitive, $\bar g$-isometric action 
of a simply-connected Lie group $\bar \G$ on $X^n$  \cite[Ch.VI, Thm.3.4]{KobNom96}.

\begin{definition}\label{def_homRFalgnonc}
  An immortal homogeneous Ricci flow solution $(M^n,g(t))_{t\in[0,\infty)}$ is called \emph{algebraically non-collapsed}, if any Cheeger-Gromov-convergent sequence of parabolic blow-downs 
\[
 g_{s_k}(t) := \tfrac{1}{s_k} \cdot g(s_k \, t)\,, \qquad s_k \to \infty, 
\]
is algebraically non-collapsed in the above sense. 
\end{definition}

We now work towards a proof of Theorem \ref{mainthm_algnonc}. Let $(\Ss, g(t))_{t\in [0,\infty)}$ be an immortal homogeneous Ricci flow solution of left-invariant metrics on a simply-connected solvable Lie group $\Ss$.  Recall that $\Ss$ is diffeomorphic to $\RR^n$.
Consider the associated bracket flow solution $(\mu(t))_{t\in[0,\infty)}$, $\mu(0) = \mu_0$ and recall that for a blown-down solution $\tfrac{1}{s} \cdot g(s t)$, $s>0$, the corresponding brackets $\mu_s(t)$ scale like
\begin{equation}\label{eqn_bracketblowdown}
    \Vert \mu_s(1) \Vert = \sqrt{s} \cdot \Vert \mu( s)\Vert.
\end{equation}
By  \cite{Bhm2014} the solution is Type-III, and we have injectivity radius estimates
 by \cite[Thm.~ 8.2]{BL17}.  
 Hence, Hamilton's compactness theorem \cite{Ham95b} 
implies that any sequence of blow-downs $\big(\Ss, (g_{s_k}(t))_{t\in[1,\infty)}  \big)_{k\in\NN}$ subconverges to a limit Ricci flow solution  in Cheeger-Gromov topology, uniformly over compact subsets of $\Ss \times [1,\infty)$.
We claim that this limit Ricci flow solution may be written as $\big(\bar \Ss, \bar g(t)_{t\in [1,\infty)} \big)$,
where $\bar \Ss$ is a simply-connected solvable Lie group, in general  not isomorphic to $\Ss$.

To that end, notice that by Theorem D.2 in \cite{BL17}  the limit is simply connected. Thus we are in a position to apply the results in \cite[$\S$6]{Heber1998} and conclude that there is a solvable Lie group of isometries acting transitively on $(\bar M^n,\bar  g)$. More precisely, one may use items (i), (ii) of Step 1 in the proof of Theorem 6.6 from that paper. A quick inspection of the proof shows that the Einstein hypothesis is not used at all for these items, and indeed all that is needed is that the limit space is simply-connected.   
By \cite[Lemma 1.2]{GrdWls}, there is also a simply-transitive solvable group of isometries, hence the limit is a solvmanifold.

\begin{lemma}\label{lem_varphibdd}
Suppose that the solution $(g(t))_{t \in [0,\infty)}$ is algebraically non-collapsed. 
Then, there exists $0< c_{\mu_0}< C_{\mu_0}$ such that $c_{\mu_0} \leq t \cdot \Vert \mu(t)\Vert^2 \leq C_{\mu_0}$ for all $t\geq 1$.
\end{lemma}

\begin{proof}
To prove the upper bound, assume on the contrary that $s_k \cdot \Vert \mu(s_k) \Vert^2 \to +\infty$ for some sequence $s_k \to \infty$. After extracting a convergent subsequence of blow-downs and using the algebraic non-collapsedness, we may apply \cite[Thm.~ 9.2]{BL17} and conclude that the corresponding brackets are bounded.  This is a contradiction: see \eqref{eqn_bracketblowdown}.

The lower bound holds even without the non-collapsedness assumption. To see that, notice that the vector field defining the bracket flow, $\mu \mapsto -\pi(A_\mu) \mu$, can be extended to a smooth vector field on $\Vs$, which is homogeneous of degree $3$. By compactness of the sphere in $\Vs$ we conclude that there is a uniform bound 
 $  \Vert \pi(A_\mu) \mu \Vert \leq C \cdot \Vert \mu \Vert^3$ for some
 $ C>0$, see also \cite[$\S$3]{scalar}. This implies that
$  \ddt \Vert \mu\Vert^2 = 2 \, \big\la \mu, \ddt \mu \big\ra \geq -  2\, C \cdot \Vert \mu \Vert^4$,
which by integrating yields 
$   \Vert \mu(t) \Vert^2 \geq  1/2 \, C \, t + \Vert \mu_0 \Vert^{-2}$
for all $t\geq0$. The desired lower bound for $t \geq 1$ now follows.
\end{proof}

Let $C_n$ denote the norm of the linear map  $\pi : \glgs \to \End(\Vs)$ defined in \eqref{eqn_gsrep}.
The next lemma says that the Ricci curvature cannot be too small for a very long time
in the algebraically non-collapsed case.

\begin{lemma}\label{lem_Riccismall}
Suppose that the solution $(g(t))_{t \in [1,\infty)}$ is algebraically non-collapsed. 
Then, there exists $\alpha_{\mu_0}>0$ such that if $t \cdot \Vert {\Ricci_{\mu(t)} }\Vert \leq 
\frac{1}{8C_n }$  holds for all $t\in [t_1, t_2]$, then $t_2 \leq \alpha_{\mu_0} \cdot t_1$.
\end{lemma}

\begin{proof}
By Lemma \ref{lem_varphibdd}, the function $\varphi:[1,\infty)\to \RR\,;\,\,t \mapsto
 t \cdot \Vert \mu(t) \Vert^2$  is bounded. Moreover, if $A_\mu := (\Riccim_\mu)_{\qg_\Beta}$ then by Cauchy-Schwarz and Remark \ref{rmk_projqbeta} we have
\[
    t \cdot \la \mu , \pi(A_\mu) \mu \ra 
    \leq C_n \cdot t \cdot \Vert \mu \Vert^2  \, \Vert A_\mu \Vert  
    \leq 2 \, C_n \cdot t \cdot \Vert \Riccim_\mu \Vert  \, \Vert \mu\Vert^2   
    \leq 2 \, C_n \cdot t \cdot  \Vert \Ricci_\mu \Vert \,  \Vert \mu \Vert^2. 
\]
Using that, for $t\in [t_1,t_2]$ we obtain
\begin{align*}
  \ddt \varphi &= \Vert \mu \Vert^2 + 2 \, t \cdot \la \mu, \ddt \mu \ra  = \Vert \mu\Vert^2 - 2 \, t \cdot \la \mu , \pi(A_\mu) \mu \ra\\
        & \geq  \Vert \mu\Vert^2 - 4 \, C_n \cdot t \cdot \Vert \Ricci_\mu \Vert \, \Vert \mu \Vert^2  \geq \unm \Vert  \mu\Vert^2 = \tfrac{1}{2\,t} \cdot \varphi.
\end{align*}
Integrating on $[t_1,t_2]$ one gets $\varphi(t_2) /\varphi(t_1)  \geq  \sqrt { {t_2}/{t_1}}$,
and the lemma follows.
\end{proof}

We now show that the blow-down limits cannot be flat.

\begin{lemma}\label{lem_Riclowerbd}
Suppose that the solution $(g(t))_{t \in [1,\infty)}$ is algebraically non-collapsed. 
Then, there exists $\delta_{\mu_0} > 0$ such that $t \cdot \Vert{ \Ricci_{\mu(t)} }\Vert \geq \delta_{\mu_0}$ for all $t\geq 1$.
\end{lemma}

\begin{proof}
Assume that this is not the case and let $s_k \to \infty$ be a sequence of times with $s_k \cdot \Vert{ \Ricci_{\mu(s_k)} }\Vert \to 0$ as $k\to\infty$. Any convergent subsequence of the corresponding sequence of blow-downs $g_{s_k}(t) = \tfrac{1}{s_k}g (s_k \, t)$ must have a Ricci-flat limit. After passing to such a subsequence, it follows 
that there exists $k_0$ such that for all $k\geq k_0$
and all  $t\in [1,1+\alpha_{\mu_0}]$ we have  $ \left\Vert  \Ricci(g_{s_k}(t))  \right\Vert 
\leq \tfrac{1}{8C_n(1+\alpha_{\mu_0}) }$.  This yields
\[
      (s_k \, t) \cdot \Vert \Ricci(g(s_k \, t))\Vert = t \cdot \left\Vert  \Ricci(g_{s_k}(t))  \right\Vert \leq \tfrac{1}{8C_n}, 
\]
for all  $ t\in [1,1+\alpha_{\mu_0}]$, thus
$\tilde t \cdot \Vert {\Ricci_{\mu(\tilde t)} }\Vert \leq  \tfrac{1}{8C_n} $ for all $\tilde t\in [s_k, (1+\alpha_{\mu_0})s_k]$. But this contradicts Lemma \ref{lem_Riccismall}.
\end{proof}

We are  finally in a position to prove Theorem \ref{mainthm_algnonc}:

\begin{proof}[Proof of Theorem \ref{mainthm_algnonc}]
Let $(\Ss, g(t))_{t\in [0,\infty)}$ be an algebraically non-collapsed Ricci flow solution of left-invariant metrics, and let $0 \neq \mu^\sg \in\sca_\Beta$ correspond to the initial metric $g(0)$. 
As in the proof of Theorem \ref{mainthm_uniq}, let $\mu(t)$ be the corresponding solution to the gauged bracket flow and $\nu(t) := \vert {\scalm ({\mu(t)}) } \vert^{-1/2} \cdot \mu(t)$ the $\scalm$-normalized solution, which after a time reparameterization solves \eqref{eqn_normalizedgaugedBF}. 

Assume that $\Vert \nu(t_k) \Vert \to \infty$ for some sequence $t_k \to \infty$, and let $\tilde \nu_k := \nu(t_k) / \Vert \nu(t_k)\Vert = \mu(t_k) / \Vert \mu(t_k) \Vert$. Then any subsequential limit $\bar \nu$ is a solvable Lie bracket with $\scalm(\bar \nu) = 0$, hence flat (see \cite[Rmk.~3.2(b)]{Heber1998}). On the other hand, $\Vert \mu(t) \Vert \sim 1/\sqrt{t}$ by Lemma \ref{lem_varphibdd}, thus
\[
		\big\Vert {\Ricci_{\mu/\Vert \mu \Vert}  }\big\Vert  \sim  \Vert \Ricci_{\sqrt{t} \cdot \mu}\Vert = t \cdot \Vert \Ricci_{\mu}\Vert  \geq \delta_{\mu_0} > 0,
\]
thanks to Lemma \ref{lem_Riclowerbd}. This implies that $\bar \nu$ cannot be flat, a contradiction.

Precisely as in the proof of Theorem \ref{mainthm_uniq} it
follows that for some subsequence $s_k \to \infty$ we have $\nu(s_k) \to \nu_{\mathsf{sol}} \in \sca_\Beta$, a solvsoliton. By Remark \ref{rem_solvsolreal}, $\nu_{\mathsf{sol}}$ is of real type. And since the nilradical of solvable brackets in $\sca_\Beta$ is of constant dimension (equal to $\rank(\Beta^+)$), for $k$ large enough $\nu(s_k)$ is of real type by Lemma \ref{lem_realtype}. Since $\mu^\sg$ and $\nu(t)$ are isomorphic for all $t$, it follows that $\Ss$ is of real type.


Conversely, assume that $\Ss$ is of real type and let $\mu(t)$ be a bracket flow solution corresponding to a Ricci flow of left-invariant metrics on $\Ss$. By Corollary 9.13 in \cite{BL17} it suffices to show that $ t \cdot \Vert \mu(t) \Vert^2 \leq C_{\mu_0}$ for some constant $C_{\mu_0} > 0$. But this
follows immediately from 
\eqref{eqn_normconvergence} and the Type-III behavior 
of homogeneous Ricci flow solutions.
\end{proof}

\section{The Einstein case}\label{sec_Einstein}

In this section we prove Theorem \ref{mainthm_Einstein}, an improvement of the above convergence results in the Einstein case, made possible by the linearization computations from Section \ref{sec_linear}.

\begin{proof}[Proof of Theorem \ref{mainthm_Einstein}]
Let $(\nu^*(t))_{t \in [0,\infty)}$ denote a  solution to the $\scalm$-normalized gauged 
bracket flow (\ref{eqn_normalizedgaugedBF}) keeping $\scalm(\nu^*(t)) \equiv -1$, with
$\nu^*(0) \in \Qb \cdot \mu^\sg$. 
By Theorem \ref{mainthm_uniq} we may assume
 that for a large time we are as close to an Einstein bracket $\mu_E$ as we like.
The set of Einstein brackets in $\Qb \cdot \mu^\sg$ with $\scalm \equiv -1$
equals $\Kb \cdot \mu_E$ by Theorem \ref{thm_uniqsolvsol} and \cite[Cor.~8.4]{GIT}. Moreover,
by Theorem \ref{thm_BFlin} for such an Einstein bracket $\mu_E$
the tangent space to its $\Slb$-orbit may be decomposed as 
$T_{\mu_E}(\Slb \cdot \mu_E) =   T_{\mu_E}(\Kb \cdot \mu_E) \oplus V_{\mu_E}$,
where $V_{\mu_E}$ denotes the sum of the eigenspaces of the linearization of \eqref{eqn_normalizedgaugedBF} with negative eigenvalues. 
 This decomposition is $\Kb$-equivariant, since the gauged bracket flow is so by Remark \ref{rmk_gbfequiv}. Using the normal exponential map of the compact orbit $\Kb\cdot \mu_E$
in $\Slb \cdot \mu_E$ in direction of $V_{\mu_E}$,
we can find coordinates $(x,y)\in U:=(1,3)^k \times (-1,1)^l$ 
of the orbit $\Slb\cdot \mu_E$ close to $\mu_E$,
where $(x,0)$ parametrizes to $\Kb$-orbit of $\mu_E$ locally and
$(0,y)$ the transversal slice given by $V_{\mu_E}$. In these coordinates
the differential equation (\ref{eqn_normalizedgaugedBF}) reads as $(x,y)' = F(x,y)=(F_1(x,y),F_2(x,y))$ with
 $F(x,0)=0$ and 
\[
   (dF)_{(x,0)} = \left( \begin{array}{cc}  0 & \tfrac{\partial F_1}{\partial y} \\
             0 & \tfrac{\partial F_2}{\partial y} \end{array} \right)_{(x,0)}\,,
\]
where $\big(\tfrac{\partial F_2}{\partial y} \big)_{(x,0)}$ has only
eigenvalues with negative real part, say bounded from the above by $-\epsilon<0$.
It is easy to see that choosing $y(0)$ small enough one can conclude that 
\begin{equation}\label{eqn_expfast}
	\nu^*(t) \underset{t\to\infty}\longrightarrow \mu_E, \qquad  \hbox{exponentially fast.}
\end{equation}
Next, consider the solution $(\nu(t))_{t\in [0,\infty)}$ to the $\scalm$-normalized  bracket flow
\begin{equation*}
  \frac{\rm d  \nu}{ {\rm d} t} = - \pi \big( \Ricci_{\nu}   + \Vert \Riccim_{\nu}\Vert^2 \cdot \Id_\sg \big) { \nu}, \qquad { \nu}(0) = \nu^*(0) \in \Qb\cdot \mu^\sg \,.
\end{equation*}
Since $\nu^*(t)$ is obtained by `gauging' $\nu(t)$, by \cite[$\S$3]{BL17} there exists a smooth family of orthogonal maps $(k(t)) \subset \Or(\sg)$ such that 
\begin{equation}\label{eqn_nugaugednu*}
	\nu(t) = k(t) \cdot \nu^*(t), \qquad  \forall \, \, t\in [0,\infty).
\end{equation}
It might be the case that the $\omega$-limit of $(\nu(t))_{t\in[0,\infty)}$ is not a single bracket. However, by \eqref{eqn_expfast} and  compactness of $\Or(\sg)$, it must be contained in $\Or(\sg) \cdot \mu_E$. In particular, since the function $\mu \mapsto \Vert {\Riccim_\mu} \Vert^2$ is $\Or(\sg)$-invariant, there exists a limit $\Vert {\Riccim_{\nu(t)} } \Vert^2 \to c_1$ as $t \to \infty$. 
And also by $\eqref{eqn_expfast}$, and using that the entries of $\Ricci_\mu$ are quadratic in the entries of $\mu$, we have that $\Ricci_{\nu^*(t)} \to c \cdot \Id_\sg$ exponentially fast, since $\mu_E$ is Einstein. Hence, from \eqref{eqn_nugaugednu*} and the $\Or(\sg)$-equivariance of $\mu \mapsto \Ricci_\mu$, we deduce that $\Ricci_{\nu(t)} \to c_2 \cdot \Id_\sg$ as $t\to \infty$, exponentially fast. We thus get exponentially fast convergence
\begin{equation}\label{eqn_vftozero}
		\Ricci_{\nu(t)} + \Vert {\Riccim_{\nu(t)} }\Vert^2 \cdot \Id_\sg  \underset{t\to\infty}\longrightarrow \alpha \cdot \Id_\sg.
\end{equation}
Taking scalar products against $\Riccim_{\nu(t)}$ we get $\alpha = 0$,
since $\scalm \equiv -1$ and 
$\la \Riccim_{\mu}, \Ricci_\mu\ra = \Vert {\Riccim_\mu}\Vert^2$ (see \cite[Lemma 2.1]{warped}).

Recall now that by Theorem \ref{thm_BFRFequiv}, for a curve
 $(h(t)) \subset \Gl(\sg)$ solving the linear equation
\begin{eqnarray}
  h'(t) \,\, = \,\, 
      -  \big(  \Ricci_{\nu(t)} + \Vert{ \Riccim_{\nu(t)} }\Vert^2 \cdot \Id_\sg  \big)
          \cdot h(t)\,, \quad h(0)=\Id_\sg, \label{Uhlnorm1}
\end{eqnarray} 
 one can recover the corresponding $\scalm$-normalized Ricci flow solution.
 By \eqref{eqn_vftozero}, 
 the differential equation  (\ref{Uhlnorm1}) can be rewritten as 
\begin{eqnarray}
h'(t)  &= &  \delta(t) \cdot h(t)
\,, \quad h(0)=\Id_\sg\,, \label{Uhlnorm2}
\end{eqnarray}
where $\delta(t)\in \Syms$ converges exponentially fast to $0$ for $t\to \infty$.
We denote by  $\sigma(t)\geq 0$ the maximum between $0$ and the largest eigenvalue of $\delta(t)$,
and by $f(t): = \tr( h(t) \cdot h(t)^T ) = \Vert h(t) \Vert ^2$.
Since $\sigma(t)$ is  integrable  over $[0,\infty)$, from
\[
  \tr( h'(t)h(t)^T )  =  \tr \big( \delta(t) \cdot h(t) \cdot h(t)^T \big),
\]
we get a differential inequality $f'(t) \leq  2 \sigma(t) f(t)$, thus
 $f(t)$ is bounded above on $[0,\infty)$.

The function $g(t) :=\det(h(t))$  satisfies an 
equation $g'(t) = g(t) \cdot s(t)$,
where $\vert s(t)\vert $ is again integrable over $[0,\infty)$.
Thus, there exists a limit $\lim_{t \to \infty}g(t)>0$. Consequently,
we can find a subsequence $(t_i)_{ i \in \NN}$ of times 
converging to infinity,
such that $\lim_{i \to \infty} h(t_i) \to h_\infty \in \Gl(\sg)$.
From \eqref{Uhlnorm2} it follows that $\Vert h'(t)\Vert$ is integrable over $[0,\infty)$, hence the curve $h : [0,\infty) \to \Gs$ has finite length and we must have $\lim_{t \to \infty}h(t)=h_\infty$.

After knowing that the $\scalm$-normalized Ricci flow has a non-flat limit, the same is also true for the scalar-curvature-normalized solution, as they differ only by scaling. The theorem now follows using the uniqueness of Einstein metrics stated in Corollary \ref{cor_uniqsolvsolmetric}.
\end{proof}

\begin{remark}\label{rem:solnonCinfty}
If the limit bracket is not Einstein but a non-trivial solvsoliton,
then the endomorphism
$\Ricci_{\nu(t)}$ $+ \Vert{ \Riccim_{\nu(t)} }\Vert^2 \cdot \Id_\sg$ converges exponentially fast to a derivation $D \neq 0$ of the limit bracket $\musol$. Thus equation (\ref{Uhlnorm2}) becomes 
\[
  h'(t) = - (\delta(t) + D) \, \cdot \, h(t), \qquad h(0) = \Id_\sg,
\] 
with $\delta(t) \to 0$.
It follows that the solution $h(t)$ does not converge.
\end{remark}

\section{The linearization of the bracket flow at a solvsoliton}\label{sec_linear}

We finally compute the linearization of the $\scalm$-normalized gauged bracket flow
\begin{equation}\label{eqn_scalmgaugedBF}
  \frac{\rm d  \nu}{ {\rm d} t} = - \pi \Big( ( \Riccim_{\nu} )_{\qg_\Beta}  + r_\nu \cdot \Id_\sg \Big) \, { \nu}, \qquad { \nu}(0) = \nu_0 \in \Qb\cdot \mu^\sg \,,
\end{equation}
at a solvsoliton bracket $\bar \mu$ which is gauged correctly w.r.t $\Beta$. Here, $r_\nu = \Vert{ \Riccim_\nu}\Vert^2$. Recall that for such a bracket $\mub \in \Vnnss \subset \sca_\Beta$ we have $\Riccim_\mub = \Beta$ and
\begin{equation}\label{eqn_betapder}
   (\Riccim_{\mub} )_{\qg_\Beta}  + r_\mub \cdot \Id_\sg  = \Beta + \Vert \Beta \Vert^2 \cdot \Id_\sg = \Beta^+ \in \Der(\mub)\,,
\end{equation}
thanks to Proposition \ref{prop_solvsol}. Thus, $\mub$ is a fixed point of \eqref{eqn_scalmgaugedBF}.

The evolution equations for $\Riccim_\mu$ stated in \cite[pp.~390]{homRF},
applied to (\ref{eqn_scalmgaugedBF}), imply that
\begin{eqnarray*}
  T_\mub \big( (\Qb\cdot \mub) \cap \{ \scalm=-1\}\big) = \big\{ \pi(A)\mub : A\in \qg_\Beta, \langle A, \Riccim_\mub \rangle = 0 \big\}\,.
\end{eqnarray*}
In particular if $\mub$ is a solvsoliton with $\Riccim_\mub = \Beta$ then 
\begin{eqnarray}\label{eqn_qbetascalone}
 T_\mub \big( (\Qb\cdot \mub) \cap \{ \scalm=-1\}  \big) = T_\mub\big(\Slb \cdot \mub\big).
\end{eqnarray}


\begin{theorem}\label{thm_BFlin}
Let $\mub \in \Vnnss \subset \sca_\Beta$ be a solvsoliton bracket with $\Riccim_\mub = \Beta$. Then, the linearization of the $\scalm$-normalized gauged bracket flow \eqref{eqn_scalmgaugedBF} at $\mub$,
\[
  L_\mub : T_\mub \big( \Slb \cdot \mub \big) \to T_\mub \big( \Slb \cdot \mub \big),
\]
has kernel given by $\pi(\kg_\Beta)\mub$, and its non-zero eigenvalues are negative.
\end{theorem}

\begin{proof}
We apply the formula for $L_\mub$ given in Lemma \ref{lem_Lmub}. Lemma \ref{lem_Pmupositive} implies that he kernel of $L_\mub$ is contained in $\pi(\kg_\Beta)\mub$. By Lemmas \ref{lem_propertiesPmub} and \ref{lem_Pmupositive}, we may now choose $A\in \slgb$ an eigenvector of both $P_\mub$ and $\ad(\Beta^+)$, with eigenvalues adding up to $c>0$. Then,
\[
   L_\mub(\pi(A)\mub) = -\pi \big(P_\mub(A) + [\Beta^+, A] \big) \mub = - c \cdot \pi(A)\mub,
\]
hence $\pi(A)\mub$ is an eigenvector of $L_\mub$ with negative eigenvalue. The theorem follows.
\end{proof}

In the rest of this section $\mub$ will denote a solvsoliton bracket as in Theorem \ref{thm_BFlin}.

\begin{lemma}\label{lem_Lmub}
If $A\in \slgb$ then $L_\mub \big( \pi(A)\mub \big) = -\pi\big( P_\mub(A) + [\Beta^+, A] \big)\mub$,
where 
\begin{equation}\label{eqn_defPmu}
    P_\mub : \slgb \to \slgb\,;  \qquad
  A \mapsto  \Big( {\rm d} \, {\Riccim} \big|_\mub (\pi(A)\mub)  \Big)_{\qg_\Beta} .
\end{equation}
\end{lemma}
\begin{proof}
Since $( \Riccim_{\mub} )_{\qg_\Beta}  + r_\mub \cdot \Id_\sg = \Beta^+$ by \eqref{eqn_betapder},
a direct computation yields
\begin{eqnarray}\label{eqn_formulaLmub}
  L_\mub \left( \pi(A)\mub \right) 
&=& 
 - \pi\left( P_\mub (A) \right) \mub 
   - \pi \left(   \left({\rm d} r |_\mub (\pi(A) \mub) \right) \cdot \Id_\sg  \right) \mub - \pi(\Beta^+) \pi(A)\mub \, .
\end{eqnarray}
Using that $\Beta^+ \in \Der(\mub)$, we obtain
\[
  \pi\big( [\Beta^+, A] \big)\mub = \pi(\Beta^+) \pi(A) \mub - \pi(A)\pi(\Beta^+)\mub = \pi(\Beta^+)\pi(A)\mub\,.
\]
On the other hand, by \eqref{eqn_Fbetanondec} we know that $\Vert {\Riccim_\mu} \Vert \geq \vert \scalm_\mu \vert \cdot \Vert \Beta \Vert$ for all $\mu \in \Qb \cdot \mub$, and at $\mub$ equality holds. Thus the first variation of $ \Vert {\Riccim_\mu} \Vert$ at $\mub$ along directions tangent to the subset of brackets with $\scalm = -1$ must vanish, and this amounts to saying that $\mub$ is a critical point for $r_\mu$ restricted to ${\Slb \cdot \mub}$. Therefore, the second term in \eqref{eqn_formulaLmub} vanishes.

Finally note, that the image of $P_\mub$ is contained in $\qg_\Beta$. But since by \eqref{eqn_ricbetapgeq0} $\mub$ is also a minimum for the functional $\mu \mapsto \la \Riccim_\mu, \Beta \ra$ restricted to $\Slb \cdot \mub$,
it follows that the image is contained in the subalgebra
$\slgb$ by its very definition.
\end{proof}

Recall from \eqref{eqn_guqbeta} that $\ad(\Beta) = \ad(\Beta^+) : \glgs \to \glgs$ is a symmetric map, and if $(\glg(\sg)_r)_{r\in \RR}$ denote its pairwise orthogonal eigenspaces, then $\hg_\Beta \subset \ggo_\Beta = \glg(\sg)_0$ and $\ug_\Beta = \bigoplus_{r>0} \glg(\sg)_r$. 

\begin{lemma}\label{lem_piAmueigenv}
For $A\in \glg(\sg)_r$ we have that $\pi(A)\mub \in V_{\Beta^+}^r$ (see paragraph before \eqref{eqn_defVnn}).
\end{lemma}
\begin{proof}
Since $\mub \in \Vzero$, 
we have  $\pi(\Beta^+) \pi(A) \mub = \pi([\Beta^+, A]) \mub + \pi(A) \pi(\Beta^+) \mub = r \cdot \pi(A) \mub.$
\end{proof}

\begin{lemma}\label{lem_propertiesPmub}
The linear maps $P_\mub$, $\ad(\Beta^+) : \slgb \to \slgb$ commute. In particular, $P_\mub$ preserves $\hg_\Beta$ and $\ug_\Beta$, and it satisfies
\[
	P_\mub (A) = 
	\begin{cases}
			\unm \cdot \big(S \circ \delta_\mub^t \delta_\mub (A) + A^t \kf_\mub + \kf_\mub A \big), & \qquad A\in \hg_\Beta; \\
			\unm \cdot \delta_\mub^t \delta_\mub (A), &\qquad A\in \ug_\Beta.
	\end{cases}
\]
Here, $S(A) := \unm (A+A^t)$, and
\[
    \delta_\mub : \glg(\sg) \to V(\sg)\,\,;\,\,\,     A \mapsto -\pi(A)\mub\,,
\]
and $\delta_\mub^t : (V(\sg),\ip) \to (\glg(\sg),\ip)$ is the usual adjoint map, 
\end{lemma}

\begin{proof}
Using the formula for ${\rm d} \, {\Riccim} \big|_\mub (\pi(A)\mub)$ given in (36) and (37) in \cite{homRF} we have 
\[
    P_\mub(A) = \unm \big( S\circ \delta_\mub^t \delta_\mub(A) \big)_{\qg_\Beta} +  \unm \big(A^t \kf_\mub + \kf_\mub A\big)_{\qg_\Beta} \, .
\]
We show that $P_\mub$ preserves the eigenspaces of $\ad(\Beta^+)$, 
recalling that by Lemma \ref{lem_Lmub} $P_\mub$ preserves $\slgb$.

First, we claim that the linear map $A \mapsto \delta_\mub^t \delta_\mub(A)$ preserves the eigenspaces of $\ad(\Beta^+)$.
Indeed, if $A_1, A_2 \in \slgb$ are eigenvectors of $\ad(\Beta^+)$ with eigenvalues $r_1\neq r_2$, then 
\[
	\la \delta_\mub^t \delta_\mub (A_1), A_2 \ra = \la \pi(A_1) \mub, \pi(A_2) \mub\ra = 0
\]  
by Lemma \ref{lem_piAmueigenv}, since two different eigenspaces of $\pi(\Beta^+)$ are orthogonal. 
For $A\in \ggo_\Beta$ this implies that $S \circ \delta_\mub^t \delta_\mub(A) \in \ggo_\Beta$, and the projection $(\cdot)_{\qg_\Beta}$ is the identity when restricted to $\ggo_\Beta$ (see Remark \ref{rmk_projqbeta}). 
For $A\in \glg(\sg)_r \subset \ug_\Beta$, $r>0$, we have that $\delta_\mub^t \delta_\mub(A) \in \glg(\sg)_r$ as well, and the map $(S(\cdot))_{\qg_\Beta}$ is the identity on $\ug_\Beta$ (see again Remark \ref{rmk_projqbeta}). The statement for the first summand thus follows. 

Regarding the second term, the decomposition $\sg = \ag \oplus \ngo$ induces natural inclusions $\glg(\ag), \glg(\ngo) \subset \glgs$. By Proposition \ref{prop_solvsol}, we have that $\glg(\ag) \subset \ggo_\beta$. Since the Killing form is trivial on the nilradical, we also have $\kf_\mub \in \glg(\ag)$, or informally, $\kf_\mub = \minimatrix{\star}{0}{0}{0}$. On the other hand, any $A\in \ug_\Beta$ is of the form $A = \minimatrix{0}{0}{\star}{\star}$, thus the map $A \mapsto (A^t \kf_\mub + \kf_\mub A)$ vanishes on $\ug_\Beta$. It clearly preserves $\ggo_\Beta$.

Finally, the formula stated for $P_\mub$ now follows form the previous discussion.
\end{proof}

\begin{lemma}\label{lem_Pmupositive}
The linear map $P_\mub : \slgb \to \slgb$ defined in \eqref{eqn_defPmu} is symmetric, positive semi-definite, and its kernel is given by $\Der(\mub) + \kg_\Beta$.
\end{lemma}

\begin{proof}
For $A\in \kg_\Beta$, by $\Or(\sg)$-equivariance of $\mu\mapsto \Riccim_\mu$ we have that 
\[
  \Riccim_{\exp(s A) \cdot \mub} = \exp(s A) \, \Riccim_\mub \, \exp(-s A) = \exp(s A) \Beta \exp(-s A) = \Beta,
\]
thus $P_\mub(\kg_\Beta) = 0$.
Recall also that since $\mub \in \Vnnss$, by \cite[Cor.~4.11]{BL17} we have $\Der(\mub) \subset \slgb$.

It remains to show that on the orthogonal complement of
$\Der(\mub) + \kg_\Beta$ in $\slgb$ 
the map $P_\mub$ is symmetric and positive definite. By Lemma \ref{lem_propertiesPmub} we may argue on $\hg_\Beta$ and $\ug_\Beta$ separately. The formula given in that lemma for $P_\mub |_{\ug_\Beta}$ immediately implies the claim in this case (recall that $\ker \delta_\mub = \Der(\mub)$).

Regarding the restriction to $\hg_\Beta$, using that $P_\mub(\kg_\Beta) = 0$ and the formula from Lemma \ref{lem_propertiesPmub} we need to worry only about the restriction $P_\mub : \pg_\Beta \to \pg_\Beta$, where $\pg_\Beta = \ggo_\Beta \cap \Sym(\sg)$. Using $\sg = \ag \oplus \ngo$, by Proposition \ref{prop_solvsol} $\pg_\beta$ decomposes as
\[
		\pg_\Beta = \Sym(\ag) \oplus \pg_\Beta^\ngo, \qquad \pg_\Beta^\ngo := \pg_\Beta \cap \glg(\ngo).
\]
Let us first see that $P_\mub$ maps these two subspaces onto orthogonal subspaces. For the Killing form term this is clear, since $\kf_\mub \in \Sym(\ag)$. 
We must thus show that for symmetric maps $A_\ag\in \glg(\ag)$ and
 $A_\ngo \in \ggo_\Beta^\ngo$ we have $\la \pi(A_\ag) \mub, \pi(A_\ngo) \mub\ra = 0$.
By linearity we may assume that the rank of $A_\ag$ is one, and that $A_\ag \, e_1 = e_1$ for some vector $e_1 \in \ag$ of norm one. Then,
\begin{align*}
  \left\la \pi(A_\ag) \mub\, , \, \pi(A_\ngo) \mub \, \right\ra =&  \,\, 2 \cdot \left\la  \ad_{\pi(A_\ag)\mub} e_1 \, , \, \ad_{\pi(A_\ngo)\mub} e_1 \right\ra \\
    =& -2 \cdot \left\la \ad_\mub e_1, [A_\ngo, \ad_\mub e_1] \right\ra \\
    =& -2 \tr A_\ngo \, [\ad_\mub e_1, (\ad_\mub e_1)^t],
\end{align*}
and the last expression vanishes since for a solvsoliton $\mub$ we have that $\ad_\mub e_1$ is a normal operator, by \cite[Theorem 4.8]{solvsolitons}. 

Having this at hand, we may prove the statement of the lemma separately for $\Sym(\ag)$ and $\pg_\Beta^\ngo$. On $A_\ngo \in \pg_\Beta^\ngo$ we have that $P_\mub = \unm S\circ \delta_\mub^t \delta_\mub$, and the claim is clear. Finally, for $A_\ag \in \Sym(\ag)$, we may apply Lemma \ref{lem_solvsolmmsol}, (i) to obtain
\[
  P_\mub(A_\ag) = \tfrac{\rm d}{\rm d s}\big|_0 \Riccim_{\exp(s A_\ag) \cdot \mub} = \tfrac{\rm d}{\rm d s}\big|_0 \exp(-s A^t_\ag)  \Riccim_\mu \exp(-s A_\ag) = - A_\ag \Beta - \Beta A_\ag.
\]
Thus, $P_\mub (A_\ag)  = 2 \cdot \Vert \Beta \Vert^2 \cdot A_\ag$, since 
 $\Beta^+|_\ag = 0$, and $A_\ag$ is an eigenvector. 
\end{proof}

 \bibliography{../bib/ramlaf2}
\bibliographystyle{amsalpha}

\end{document}